\documentclass[11pt,letterpaper,reqno]{amsart}
\usepackage[margin=1in]{geometry}
\usepackage{color}
\usepackage[latin1]{inputenc}
\usepackage{amsmath,amsfonts,amsthm,epsfig,graphicx,amssymb,bm}
\usepackage[colorlinks=true, linkcolor=blue, citecolor=blue]{hyperref}
\usepackage{cleveref}
\usepackage{esint}
\usepackage{caption}
\usepackage{enumitem}
\usepackage{overpic}
\usepackage{aligned-overset}

\newcommand{\E}{\mathcal{E}}
\newcommand{\F}{\mathcal F}

\newcommand{\X}{\mathcal{X}}

\newcommand{\R}{\mathbb{R}}

\newcommand{\e}{\varepsilon}

\newcommand{\s}{\sigma}

\newcommand{\loc}{{\rm loc}}

\newcommand{\spt}{{\rm spt}}

\newcommand{\pa}{\partial}

\newcommand{\cc}{\subset\!\subset}

\newcommand{\cl}{\mathrm{cl}\,}

\newcommand{\p}{\mathbf{p}}

\newcommand{\m}{\mathbf{m}}
\newcommand{\var}{\mathbf{var}\,}

\newcommand{\ehn}{\overset{\hn}{=}}

\newcommand{\ehone}{\overset{\mathcal{H}^1}{=}}
\newcommand{\shk}{\overset{\mathcal{H}^k}{\subset}}
\newcommand{\shn}{\overset{\mathcal{H}^{n-1}}{\subset}}
\newcommand{\xbub}{\X_{\rm bub}^m}

\newcommand{\rn}{\mathbb{R}^{n}}

\newcommand\restr[2]{{
  \left.\kern-\nulldelimiterspace 
  #1 
  \right|_{#2} 
  }}

\newcommand{\one}{{\scriptscriptstyle{(1)}}}
\newcommand{\zero}{{\scriptscriptstyle{(0)}}}
\newcommand{\half}{{\scriptscriptstyle{(1/2)}}}

\newcommand{\ste}{S}
\newcommand{\ovx}{{\overline{x}}}
\newcommand{\oner}{{\scriptscriptstyle{(1)_{\mathbb{R}}}}}
\newcommand{\zeror}{{\scriptscriptstyle{(0)_{\mathbb{R}}}}}
\newcommand{\pastr}{\pa^*_{\mathbb{R}}}
\newcommand{\hone}{\mathcal{H}^1}
\newcommand{\onerk}{{\scriptscriptstyle{(1)_{\mathbb{R}^k}}}}
\newcommand{\zerork}{{\scriptscriptstyle{(0)_{\mathbb{R}^k}}}}
\newcommand{\pastrk}{\pa^*_{\mathbb{R}^k}}

\newcommand{\pcprime}{P_{{\bf c}'}}

\newcommand{\mres}{\mathbin{\vrule height 1.6ex depth 0pt width 
0.13ex\vrule height 0.13ex depth 0pt width 1.3ex}}

\theoremstyle{plain}
\newtheorem{theorem}{Theorem}[section]
\newtheorem{lemma}[theorem]{Lemma}
\newtheorem{corollary}[theorem]{Corollary}

\newtheorem*{theorem*}{Theorem}
\newtheorem*{corollary*}{Corollary}

\theoremstyle{definition}
\newtheorem{definition}[theorem]{Definition}

\newtheorem{remark}[theorem]{Remark}

\newtheorem*{notation*}{Notation}

\numberwithin{equation}{section}
\numberwithin{figure}{section}

\renewcommand{\m}{\mathbf{m}}
\newcommand{\hn}{\mathcal{H}^{n-1}}
\newcommand{\xl}{\X_{\rm lens}}
\newcommand{\pc}{P_{\bf c}}
\newcommand{\rl}{r_{\textup{lens}}}
\newcommand{\ee}{\mathbf{e}}

\allowdisplaybreaks

\title{An Infinite Double Bubble Theorem}
\subjclass[2010]{49Q20, 49Q05}

\author{Lia Bronsard}
\address{Department of Mathematics \& Statistics, McMaster University, Hamilton ON L8S 4L8, Canada}
\email{bronsard@mcmaster.ca}
\author{Michael Novack}
\address{Department of Mathematical Sciences, Carnegie Mellon University, Wean Hall 6113, Pittsburgh, PA 15213, United States of America}
\email{mnovack@andrew.cmu.edu}

\begin{document}

\begin{abstract} 
The classical double bubble theorem characterizes the minimizing partitions of $\rn$ into three chambers, two of which have prescribed finite volume. In this paper we prove a variant of the double bubble theorem in which two of the chambers have infinite volume. Such a configuration is an example of a {\it(1,2)-cluster}, or a partition of $\rn$ into three chambers, two of which have infinite volume and only one of which has finite volume \cite{AlaBroVri23}. A $(1,2)$-cluster is locally minimizing with respect to a family of weights $\{c_{jk}\}$ if for any $B_r(0)$, it minimizes the interfacial energy $\sum_{j<k} c_{jk} \hn(\pa \X(j) \cap \pa\X(k) \cap B_r(0))$ among all variations with compact support in $B_r(0)$ which preserve the volume of $\X(1)$. For $(1,2)$ clusters, the analogue of the weighted double bubble is the {\it weighted lens cluster}, and we show that it is locally minimizing. Furthermore,
under a symmetry assumption on $\{c_{jk}\}$ that includes the case of equal weights, the weighted lens cluster is the unique local minimizer in $\mathbb{R}^n$ for $n\leq 7$, with the same uniqueness holding in $\mathbb{R}^n$ for $n\geq 8$ under a natural growth assumption. We also obtain a closure theorem for locally minimizing $(N,2)$-clusters.


\end{abstract}

\maketitle

\setcounter{tocdepth}{2}


\section{Introduction}\label{sec:intro}

The classical cluster problem in $\rn$ is to find the configuration of $N$ regions of prescribed finite volumes and an exterior region that minimizes the total area of the interfaces between regions \cite{Mor16}. Variants include the immiscible fluid problem, in which the interfaces between pairs of regions are weighted by coefficients that depend on the pair. The existence of minimizers for a general class of problems of this type has been proved by Almgren in \cite{Alm76}. This existence opens the door to the analysis and possible characterization of their shape, an issue which has been extensively studied but still presents many interesting open questions.

\medskip

When $N \leq n+1$ and the weights on the interfaces are equal, the canonical configuration of $N$ bounded chambers and an exterior region known as the {\it standard $N$-bubble} is a natural candidate for minimality \cite[Problem 2]{SulMor96}. The $N$-bubble conjecture states that the standard $N$-bubble is the unique minimizer of the equal weights energy. If $N=1$, this reduces to the isoperimetric problem. For any pair of volumes, the double bubble conjecture was fully settled first in the plane by Foisy-Alfaro-Brock-Hodges-Zimba in \cite{FoiAlfBroHodZim93}, next in $\mathbb{R}^3$ by Hutchings-Morgan-Ritor\'{e}-Ros \cite{HutMorRitRos02}, then in $\mathbb{R}^4$ by Reichardt-Heilmann-Lai-Spielman \cite{ReiHeiLaiSpi03}, and finally in $\rn$ for arbitrary $n$ by Reichardt \cite{Rei08}. We refer the reader to the references therein for additional results that also contributed to the complete resolution of the double bubble conjecture. The triple bubble conjecture was verified in $\mathbb{R}^2$ by Wichiramala \cite{Wic04}. Recently, Milman and Neeman have shown in \cite{MilNee22} that $N$-bubbles are the unique minimizers when $N\leq \min \{4, n\}$ and in \cite{MilNee23} that quintuple bubbles are minimizers when $n\geq 5$. Also, minimizers for the equal volumes quadruple bubble problem in $\mathbb{R}^2$ were characterized by Paolini and Tortorelli \cite{PaoTor20}. The characterization of minimizers in the immiscible fluids problem with non-equal weights has received less attention, having only been resolved in the case of two chambers by Lawlor \cite{Law14}; see also \cite[Theorem 7.2]{HutMorRitRos02}. 

\medskip

In classical clusters, there is a single chamber with infinite volume. In \cite{AlaBroVri23}, Alama-Bronsard-Vriend have generalized this concept through the introduction of $(N,M)$ clusters, which partition $\rn$ into $N$ chambers of finite prescribed volume and $M$ chambers of infinite volume, with $M$ allowed to be strictly greater than one \cite{AlaBroVri23}. This requires changing the definition of minimality, since any $(N,M)$-cluster when $M\geq 2$ automatically has infinite energy on all of space. One possible notion is thus {\it local minimality} subject to compactly supported perturbations which preserve the volume of the non-infinite chambers. The generalization of the standard weighted double bubble is the {\it standard weighted lens cluster}, which the same authors showed is the unique local minimizer for the equal weights energy in the plane. The equal weights problem was further studied by Novaga-Paolini-Tortorelli \cite{NovPaoTor23}, who obtained a general closure theorem for local minimizers when the limiting $(N,M)$ cluster has flat interfaces outside some compact set. The closure theorem allows for the construction of several locally minimizing $(N,M)$-clusters in any dimension. They also proved that any local minimizer in the plane necessarily has at most $3$ chambers with infinite volume and completely characterized planar local minimizers when $N+M \leq 4$.

\medskip

In this paper we study $(1,2)$-clusters $\{\X(j)\}_{j=1}^3$, where $|\X(1)|<\infty$, and the weighted energy $\sum_{j<k}c_{jk}\hn(\pa^* \X(j) \cap \pa^* \X(k))$, where the family of weights satisfies standard positivity and triangularity conditions, and first show that the standard weighted lens cluster is locally minimizing (Theorem \ref{thm:local minimality of lens}). Our main result (Theorem \ref{thm:double bubble}) can be summarized as follows:

\begin{center}
    If $c_{12}=c_{13}$, then, up to rigid motions of $\rn$, the standard weighed lens cluster is the unique locally minimizing $(1,2)$-cluster in $\rn$ for $n\leq 7$. The same uniqueness holds in $\rn$ for $n\geq 8$ among locally minimizing $(1,2)$-clusters with planar growth at infinity.
\end{center}

\noindent The planar growth assumption means that asymptotically, the two infinite volume chambers have the same energy density as a pair of complementary halfspaces. This assumption is natural in light of the existence of singular perimeter minimizers in higher dimensions; see Remark \ref{remark:n>7}. The proof of Theorem \ref{thm:double bubble} is based on an energy comparison argument with the standard weighted lens cluster and the geometric rigidity entailed by the resulting energy identity. The comparison is made possible by decay properties of exterior minimal surfaces with planar growth; see Section \ref{subsec:discussion} for a more detailed discussion of the proof. We also include 
a proof of the equivalence between the above notion of local minimality and another natural one (Lemma \ref{lemma:equiv min notions}) and use it to
strengthen the closure theorem from \cite{NovPaoTor23} when there are two chambers with infinite volume (Corollary \ref{corollary:closure theorem}).

\begin{figure}
\begin{overpic}[scale=0.65,unit=1mm]{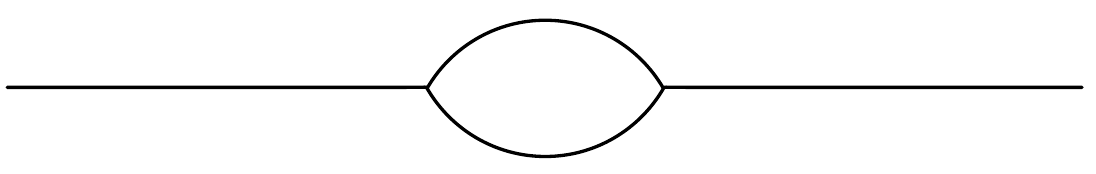}
\put(75,12.5){\small{$\xl(2)$}}
\put(45,7){\small{$\xl(1)$}}
\put(75,1.5){\small{$\xl(3)$}}
\end{overpic}
\caption{The standard weighted lens cluster $\xl$.}\label{fig:lens}
\end{figure}

\medskip
The partitioning problem for $(N,M)$ clusters introduced in \cite{AlaBroVri23} comes from the study of triblock copolymer models in the 2D torus by Alama-Bronsard-Lu-Wang in \cite{AlaBroLuWanprep}. In \cite{AlaBroLuWan22}, these authors study a partitioning problem in the torus with three phases, one of which occupies nearly all of the total area with the other two accounting for only a tiny fraction. In that case, the global  minimizers form weighted double-bubble or core-shell patterns depending on the parameters. In \cite{AlaBroLuWanprep}, they
consider this partitioning problem but with two of the phases occupying nearly all of the total area and the third accounting for only a tiny fraction. In that case, we expect minimizers of the nonlocal triblock copolymer energy to form a lamellar pattern with the two majority phases, with tiny droplets of the third phase aligned on each lamellar stripe.  By blowing up the droplets in the limit of vanishing area of the third phase we expect to recover the lens shape in the plane, and the characterization of the  \textit{vesica piscis} as the unique minimizer (up to symmetries) of this problem is critical to the analysis of the small-area limit of the triblock problem.

\medskip

The rest of the paper is organized as follows. In Section \ref{sec:results}, we give precise definitions and statements of our results, and Section \ref{sec:notat and prelim} contains some background results. Section \ref{subsec:prelim} establishes various useful properties of locally minimizing $(N,M)$-clusters. Some of these may be applicable in future investigations of $(N,M)$-clusters, and so we have opted for generality regarding weights/number of chambers/etc.~ when possible in this part. Finally, we prove Theorem \ref{thm:local minimality of lens} and Theorem \ref{thm:double bubble} in Section \ref{sec:main proof}.

\section{Results}\label{sec:results}

\subsection{General {\bf (\textit{N},\textit{M})}-clusters}

We begin by recalling the notion of an $(N,M)$-cluster due to Alama-Bronsard-Vriend \cite{AlaBroVri23} and some basic terminology for these objects. Definitions \ref{def:nm cluster}-\ref{def:interfaces} can be found in \cite[Section 1]{AlaBroVri23}, and Definitions \ref{def:weighted perimeter}-\ref{def:loc mins} are the weighted analogues of the energy and local minimizers in \cite{AlaBroVri23}.

\begin{definition}[$(N,M)$-clusters]\label{def:nm cluster}
    An $(N,M)${\it-cluster} $(\E,\F)$ is a pair of finite families of sets of locally finite perimeter
\begin{align}\notag
\E = \{ \E(h) \}_{h=1}^N\,,\quad \F = \{ \F(i)\}_{i=1}^M    
\end{align}
such that, denoting by $|\cdot|$ the Lebesgue measure,
\begin{enumerate}[label=(\roman*)]
    \item $0<|\E(h)|<\infty$ for $1\leq h \leq N$,
    \item $|\F(i)|=\infty$ for $1\leq i \leq M$,
    \item any two distinct sets from either family are Lebesgue disjoint, that is
    \begin{align*}
        &|\E(h) \cap \E(h')|=0\quad \textup{for all $1\leq h < h' \leq N$}\,, \\
        &|\F(i) \cap \F(i')|=0\quad \textup{for all $1\leq i < i' \leq M$}\,,\quad\textup{and}\\
        &|\E(h) \cap \F(i)| = 0 \quad \textup{for all $1\leq h \leq N$ and $1\leq i \leq M$}\,,\quad\textup{and}
    \end{align*} 
    \item $\big| \mathbb{R}^n \setminus (\cup_{h=1}^N \E(h) \cup \cup_{i=1}^M \F(i))| = 0$.
\end{enumerate}
The sets with finite Lebesgue measure are called {\it proper chambers} and the sets with infinite Lebesgue measure are called {\it improper chambers}.
\end{definition}

\noindent We will streamline the notation by referring to an $(N,M)$-cluster $(\E,\F)$ by the single family $\X=\{\X(j) \}_{j=1}^{N+M}$, where $\X(j)=\E(j)$ if $1\leq j \leq N$ and $\X(j+N) = \F(j)$ if $1 \leq j \leq M$.

\begin{definition}[Volume vector]
    For an $(N,M)$-cluster $\X$, the {\it volume vector} is the $(N+M)$-tuple
\begin{align}\notag
    \m(\X) = (|\X(1)|,\dots,|\X(N+M)|)
\end{align}
    with entries in the extended positive real numbers $(0,\infty]$.
\end{definition}

\begin{definition}[Interfaces]\label{def:interfaces}
    For an $(N,M)$-cluster $\X$, the {\it interfaces} are the locally $\hn$-rectifiable sets formed by the intersection of the reduced boundaries of two distinct chambers, so
\begin{align}\notag
    \X(j,k) = \pa^* \X(j) \cap \pa^* \X(k)\qquad \forall  1\leq j < k \leq N+M\,.
\end{align}
\end{definition}

\noindent Since an $(N,M)$-cluster $\X$ corresponds to a Caccioppoli partition of $\rn$, the Lebesgue points of $\X(j)$ and the interfaces partition $\rn$ up to a $\hn$-null set (see e.g. \cite[Theorem 4.17]{AmbFusPal00}), that is

\begin{align}\label{eq:hn partition}
   \hn \Bigg( \rn \setminus \Bigg[\bigcup_{1\leq j\leq N+M} \X(j)^\one \cup \bigcup_{1\leq j<k\leq N+M} \X(j,k)\Bigg]\Bigg)=0\,,
\end{align}
and any pair of sets from either union are disjoint.

\begin{definition}[Relative (weighted) perimeter]\label{def:weighted perimeter}
    For an $(N,M)$-cluster $\X$ and family of weights ${\bf c}:=\{c_{jk}\}_{1\leq j<k\leq N+M}$, the {\it $\bf c$-perimeter of $\X$ relative to a Borel set $B\subset \rn$} is
\begin{align}\label{eq:weighted cluster energy}
    \pc(\X;B) = \sum_{1\leq j < k \leq N+M} c_{jk}\hn( \X(j,k)\cap B)\,.
\end{align}
\end{definition}

\begin{definition}[$\bf c$-locally minimizing clusters]\label{def:loc mins}
    A {\it $\bf c$-locally minimizing $(N,M)$-cluster} $\X$ associated to a family $\bf c$ of weights satisfies $\cl \pa^* \X(j) = \pa \X(j)$ for all $1\leq j \leq N+M$ (following the convention for minimizing clusters in \cite[Part IV]{Mag12}) and, for every $r>0$, the local minimality condition
\begin{align}\label{eq:min condition}
    \pc(\X;B_r(0)) \leq \pc(\X';B_r(0))
\end{align}
for every $(N,M)$-cluster $\X'$ such that $\m(\X) = \m(\X')$ and $\X(j) \Delta \X'(j) \cc B_r(0)$ for each $1\leq j \leq N+M$.
\end{definition}



\subsection{Statements}
To state the main theorem, we must first discuss the behavior of $\bf c$-locally minimizing $(1,2)$-clusters at infinity; the proofs of the following statements are included in Section \ref{subsec:prelim}. Let us assume that the family of weights ${\bf c}=\{c_{12},c_{13},c_{23}\}$ satisfies
\begin{align}\tag{pos.}\label{eq:pos}
    & \min\{c_{12},c_{13},c_{23}\}>0,\\ \tag{$\triangle$-ineq.} \label{eq:triangle}
    &c_{12}< c_{13}+c_{23},\quad c_{13}<c_{12}+c_{23},\quad \textup{and}\quad c_{23}<c_{13}+c_{23}\,.
\end{align}
Comparison arguments based on \eqref{eq:pos}-\eqref{eq:triangle} show that if $\X$ is a $\bf c$-locally minimizing $(1,2)$-cluster, then $\X(1)\cc B_R(0)$ for some $R<\infty$. In particular, the volume constraint is trivially satisfied when testing the minimality \eqref{eq:min condition} with $\X'$ such that $\X(1)=\X'(1)$, which is the case for any $\X'$ which only differs from $\X$ on $\rn \setminus \cl B_R(0)$. As a consequence, $\X(2)$ and $\X(3)$ are perimeter minimizers in $\rn \setminus \cl B_R(0)$, with common boundary satisfying the minimal surface equation distributionally, so the monotonicity formula allows us to compute the energy density of $\X$ at infinity.

\begin{definition}[Density at infinity]\label{def:growth def}
    If the family $\bf c$ satisfies \eqref{eq:pos}-\eqref{eq:triangle} and $\X$ is a $\bf c$-locally minimizing $(1,2)$-cluster, then the {\it density of $\X$ at infinity} is
\begin{align}\label{eq:density at infinity}
\Theta_\infty(\X) = \lim_{r\to \infty}\frac{\pc(\X;B_r(0))}{r^{n-1}}\,.
\end{align} 
We say $\X$ has {\it planar growth at infinity} if $\X$ is a $\bf c$-locally minimizing $(1,2)$-cluster and $\Theta_\infty(\X) = c_{23}\,\omega_{n-1}$, where $\omega_{n-1}$ is the $(n-1)$-dimensional measure of the unit ball in $\mathbb{R}^{n-1}$.
\end{definition}




\medskip

Next we define the standard weighted lens cluster; see Figure \ref{fig:lens}.

\begin{definition}[Standard weighted lens cluster]
    The {\it standard weighted lens cluster} $\xl$ in $\rn$ associated to weights $\{c_{12},c_{13},c_{23}\}$ satisfying \eqref{eq:pos}-\eqref{eq:triangle} is the $(1,2)$-cluster such that
\begin{enumerate}[label=(\roman*)]
    \item $|\xl(1)|=1$,
    \item $\pa \xl(1)=\cl \xl(1,2) \cup \cl \xl(1,3)$, where $\xl(1,2)\subset \{x_n>0\}$ and $\xl(1,3)\subset \{x_n<0\}$ are spherical caps with boundaries contained in the plane $\{x_n=0\}$,
    \item $\xl(2,3)=\{x_n=0\}\setminus D$, where $D\subset \{x_n=0\}$ is the closed $(n-1)$-dimensional disk with boundary (relative to $\{x_n=0\}$) given by $\pa D = \{x_n=0\} \cap \pa \xl(1)$, and
    \item for each $j$, the two smooth surfaces forming $\partial \xl(j)$ and meeting along $\pa D$ form an angle $\theta_j\in (0,\pi)$, and 
\begin{align}\notag
\frac{\sin \theta_1}{c_{23}}=\frac{\sin \theta_2}{c_{13}}=\frac{\sin\theta_3}{c_{12}}\,,    
\end{align}
where $\theta_1+\theta_2+\theta_3=2\pi$. \end{enumerate}
\end{definition}

\noindent The following theorems establish the $\bf c$-local minimality of the standard weighted lens cluster as well as its uniqueness under the additional symmetry assumption $c_{12}=c_{13}$ and a natural growth assumption in higher dimensions. By scaling, they may also be applied to $(1,2)$ clusters with volume vector $(m,\infty,\infty)$ for any $m\in (0,\infty)$.

\begin{theorem}[Local minimality of the standard weighted lens cluster]\label{thm:local minimality of lens}
    If the family $\bf c$ of weights satisfies \eqref{eq:pos}-\eqref{eq:triangle}, then the standard weighted lens cluster $\xl$ is a $\bf c$-locally minimizing $(1,2)$ cluster in $\rn$.
\end{theorem}

\begin{theorem}[Weighted double bubble theorem for $(1,2)$-clusters]\label{thm:double bubble}
    If $1\leq n \leq 7$ and the family $\bf c$ of weights satisfies \eqref{eq:pos}-\eqref{eq:triangle} and $c_{12}=c_{13}$, then, up to rigid motions of $\rn$, the standard weighted lens cluster $\xl$ is the unique $\bf c$-locally minimizing $(1,2)$-cluster with volume vector $(1,\infty,\infty)$. If $n\geq 8$, the same uniqueness holds among $\bf c$-locally minimizing $(1,2)$-clusters that also have planar growth at infinity.
\end{theorem}

\noindent Theorem \ref{thm:double bubble} characterizes all $\bf c$-locally minimizing $(1,2)$-clusters in $\rn$ when $c_{12}=c_{13}$ and $1\leq n \leq 7$. In $\rn$ for $n\geq 8$, there may be other $\bf c$-locally minimizing $(1,2)$-clusters.

\begin{remark}[Other $\bf c$-locally minimizing clusters in $\R^n$ for $n\geq 8$]\label{remark:n>7}
   The blow-down of any $\bf c$-locally minimizing $(1,2)$-cluster $\X$ in $\rn$ corresponds to a perimeter minimizer in all of space. When $n\leq 7$, this blow-down necessarily corresponds to a halfspace, as there are no other entire perimeter minimizers. However when $n\geq 8$, the blow-down might be a singular minimizing cone. By Allard's theorem, the blow-down cone is singular and non-planar if and only if $\Theta_\infty(\X) >c_{23} \omega_{n-1}$. It seems possible that locally minimizing $(1,2)$-clusters can be constructed using a cone such as the Simons cone in $\R^8$ to define $\X(2)$ and $\X(3)$ outside $\X(1)$, in which case the planar growth restriction is optimal. 
\end{remark}

\begin{remark}[Uniqueness of the lens under a different notion of minimality]
    An alternative notion of local minimality for a $(1,2)$-cluster $\X$ to Definition \ref{def:loc mins} would be for \eqref{eq:min condition} to hold on every $B_r(0)$ among those $\X'$ such that $\X(j) \Delta \X'(j) \cc B_r(0)$ and $|\X(j) \cap B_r(0)|=|\X'(j) \cap B_r(0)|$ for each $1\leq j \leq 3$. The lens cluster $\xl$ is certainly minimal under this definition since it is less restrictive than Definition \ref{def:loc mins}, in that it requires minimality against fewer competitors. However, these two notions are actually equivalent; see Lemma \ref{lemma:equiv min notions}. Thus the symmetric lens cluster is the unique local minimizer under this definition as well by Theorem \ref{thm:double bubble}. As a corollary to the equivalency of these minimality notions, the closure theorem from \cite{NovPaoTor23} can be strengthened in the case of two improper chambers by removing the asymptotic flatness assumption; see Corollary \ref{corollary:closure theorem}.
\end{remark}

\begin{remark}[Connection with large-volume exterior isoperimetry]\label{remark:residue connection}
 In \cite{MagNov22}, for a compact set $W$, the second author and F. Maggi studied the limit as $v\to \infty$ of minimizers for the exterior isoperimetric problem
\begin{align}\notag
  \min \{P(E;\rn \setminus W): |E|=v,\,\, E \subset \rn \setminus W \}\,.  
\end{align}
The limiting object as $v\to \infty$ of the boundaries $\pa E_v$ of a sequence $\{E_v\}_{v>0}$ of minimizers is an exterior minimal surface in $\rn \setminus W$. A key tool in obtaining a sharp geometric description of $\pa E_v$ for large $v$ is the analysis of fine properties of this minimal surface, in particular uniqueness of blow-down cones and asymptotic decay. The properties of the interface $\pa^* \X(2) \cap \pa^* \X(3)$, which is an exterior minimal surface, play a similar role in our analysis. This connection is visible in for example Corollary \ref{corollary:asymptotic expansion} and the beginning of the proof of Theorem \ref{thm:double bubble}, in which asymptotic properties of $\pa^* \X(2) \cap \pa^* \X(3)$ are important. 
\end{remark}

\subsection{Discussion}\label{subsec:discussion}
The proof of Theorem \ref{thm:local minimality of lens} utilizes the minimality of the weighted double bubbles as shown by Lawlor \cite{Law14}. When $c_{12}=c_{13}$, an alternate proof using the symmetry is also available; see Remark \ref{remark:alternate proof of minimality}. The proof of Theorem \ref{thm:double bubble} requires showing that given $\bf c$ with $c_{12}=c_{13}$, any $\bf c$-locally minimizing $(1,2)$-cluster $\X$ must in fact be equivalent via null sets and rigid motions to $\xl$. The starting point is using the planar growth and the uniqueness of blow-downs/asymptotic expansion from Maggi-N.~ \cite{MagNov22} to show that outside some compact set, $\pa^* \X(2) \cap \pa^* \X(3)$ is the graph, say over $H:=\{x\in \rn:x_n=0\}$ of an exterior solution $u$ to the minimal surface equation decaying very fast to $H$. The hyperplane $H$ is thus our candidate for the plane over which we expect $\X$ to be symmetric, in the sense that $\X(1)$ is symmetric over $H$ and $\pa^* \X(2) \cap \pa^* \X(3) = H \setminus \cl \X(1)$. In fact, if we knew that $\X$ possessed this symmetry, then we could use the uniqueness of minimizers for the classical liquid drop problem in a halfspace and the relationship between $\pc$ and the liquid drop energy (Lemma \ref{lemma:rewriting cluster energy}) to conclude that, up to translations along $H$, $\X = \xl$. So to prove the uniqueness of the lens, it remains to show that $\X$ is symmetric over $H$. For the proper chamber $\X(1)$, one possible way of obtaining such symmetry is via a rigidity result for the Steiner inequality due to Barchiesi-Cagnetti-Fusco \cite{BarCagFus13}; for the improper chambers, symmetry, or flatness, should come from the fact that the projection of $\pa^* \X(2) \cap \pa^* \X(3)$ onto $H$ strictly decreases $\hn$ measure if $\nu_{\X(2)}$ deviates from $- e_n$; see Figure \ref{fig:symmetrization}. We therefore prove a local Steiner-type inequality for $(1,2)$ clusters in Lemma \ref{lemma:proj and symm} -- it is here we use the symmetry assumption $c_{12}=c_{13}$ -- and would like to compare the energy of the ``symmetrized" $\X$ to $\X$ itself.
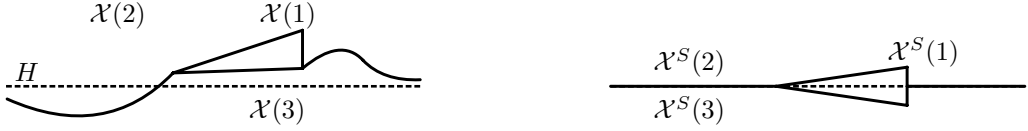
\begin{figure}
\begin{overpic}[scale=0.7,unit=1mm]{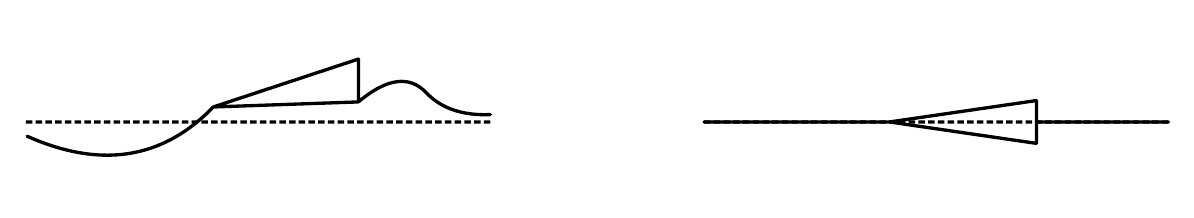}
\put(10,14){\small{$\X(2)$}}
\put(25,5){\small{$\X(3)$}}
\put(26,14){\small{$\X(1)$}}
\put(63,5){\small{$\X^S(3)$}}
\put(63,9.5){\small{$\X^S(2)$}}
\put(85,10.5){\small{$\X^S(1)$}}
\put(3,8.5){\small{$H$}}
\end{overpic}
\caption{On the left is the original cluster $\X$, and on the right is the ``symmetrized" cluster $\X^S$ over the dotted line $H$. Unless $\X$ is planar outside some compact set, $\X^S$ will not be a compactly supported variation of $\X$. If $\pa^* \X(2) \cap \pa^* \X(3)$ decays fast enough to $H$, then $\pc(\X;B_r)$ and $\pc(\X^S;B_r)$ can be compared up to small error.}\label{fig:symmetrization}
\end{figure}
This is in general not allowable, since the symmetrization will not be a compactly supported variation of $\X$. However, the decay of $\pa^*\X(2) \cap \pa^* \X(3)$ to $H$ makes such a comparison possible on $B_r(0)$ up to a small error of order $r^{-1}$. By sending $r\to \infty$ and making use of a monotonicity property of the energy gap between $\X$ and its symmetrization, we are able to conclude that, up to translating $\X$, $\pc (\X;B_r(0)) = \pc(\xl;B_r(0))$ for all $r>0$. At this point we can finally show that $\X$ must be symmetric over $H$ due to the aforementioned rigidity considerations.

\section{Notation and Preliminaries}\label{sec:notat and prelim}
\subsection{Notation}
Let $C_r$ be the infinite cylinder
\begin{align}\notag
C_r = \{x\in \rn : x_1^2 + \cdots + x_{n-1}^2 < r^2 \}\,.
\end{align}
We will write $\overline{x}$ to distinguish points in $\mathbb{R}^{n-1}$ from points in $\rn$, and set $B_r^{n-1}(\overline{x})$ to be the $(n-1)$-dimensional ball of radius $r$ centered at $\overline{x}\in \mathbb{R}^{n-1}$. 

\medskip

For Borel sets $A,B\subset \mathbb{R}^n$ and $1\leq k \leq n$, we write $A\overset{\mathcal{H}^k}{=} B$ when $\mathcal{H}^k(A \Delta B)=0$
and $A\shk B$ when $\mathcal{H}^k(A \setminus B)=0$.

\medskip

We will adhere to common notation regarding sets of finite perimeter; see for example the book \cite{Mag12}. For a set $E\subset \rn$ of locally finite perimeter $P(E;B)$ is the perimeter of $E$ inside $B$ if $B\subset \rn$ is Borel, $E^{\scriptsize{(t)}}$ is the set of points of Lebesgue density $t\in [0,1]$, $\pa^e E = \rn \setminus (E^\one \cup E^\zero)$ is the essential boundary of $E$, and $\pa^* E$ is the reduced boundary with outer (measure-theoretic) normal $\nu_E:\pa^* E \to \mathbb{S}^{n-1}$. When we are using these concepts on some $k$-dimensional set of locally finite perimeter $F\subset \mathbb{R}^k$ where $k\neq n$, we will often write $F^\onerk$, $F^\zerork$, and $\pastrk F$ to emphasize that these operations are taken with respect to $\mathbb{R}^k$. The $L^1_\loc$-convergence of the characteristic functions $\mathbf{1}_{E_m}$ of sets of finite perimeter to some $\mathbf{1}_E$ will be denoted by $E_m \overset{\loc}{\to}E$.

\medskip

To describe the convergence of clusters, we say that {\it $\X_m$ locally converges to $\X$}, or
\begin{align}\notag
    \X_m \overset{\loc}{\to} \X\,,
\end{align}
if, for each $1\leq j \leq N+M$,
\begin{align}\notag
   \big|(\X_m(j) \Delta  \X(j))\cap K\big|\to 0\qquad \forall K \cc \rn\,.
\end{align}

\medskip

Lastly, we remark that given weights $\{ c_{12},c_{13},c_{23}\}$ and corresponding $c_1$, $c_2$, $c_3$ defined via the linear system $c_{jk}=c_j + c_k$, \eqref{eq:triangle} is equivalent to 
\begin{equation}\label{eq:pos for alt form}
    \min\{c_1,c_2,c_3\}>0\,,
\end{equation} 
and, for any Borel $B\subset \rn$,
\begin{align}\label{eq:pc rewrite}
    \pc(\X;B) &= c_1P(\X(1);B) + c_2P(\X(2);B) + c_3P(\X(3);B)\,.
\end{align}

\subsection{Preliminaries}\label{subsec:sfop prelim} Here we collect some facts and theorems from the literature that will be used in the proofs.

\medskip

First is Federer's theorem \cite[Theorem 3.61]{AmbFusPal00}, which says that if $E\subset \rn$ is a set of locally finite perimeter, then
\begin{align}\label{eq:fed thm}
    \pa^* E \subset E^\half \subset \pa^e E\,,\quad \pa^e E \ehn \pa^* E\,,\quad \mbox{and}\quad \rn \ehn E^\one \cup E^\zero \cup E^\half\,.
\end{align}
Second are the formulas for perimeters of unions and intersections \cite[Theorem 16.3]{Mag12}. For any Borel $G\subset \rn$ and sets of locally finite perimeter $E$ and $F$, we have
\begin{align}\label{eq:cut and paste union}
    P(E\cup F;G) &= P(E;F^\zero \cap G) + P(F;E^\zero \cap G) + \hn(\{\nu_E = \nu_F\}\cap G)\quad\textup{and} \\ \label{eq:cut and past intersection}
    P(E\cap F;G) &= P(E;F^\one \cap G) + P(F;E^\one \cap G) + \hn(\{ \nu_E=\nu_F\}\cap G)\,.
\end{align}
As a consequence of \eqref{eq:cut and paste union}-\eqref{eq:cut and past intersection} and \eqref{eq:fed thm}, if $E$ and $F$ are sets of finite perimeter and $H$ is a set of locally finite perimeter such that $\hn(\partial^* H \cap (\partial^* E \cup \partial^* F))=0$, then for every Borel $G\subset \rn$,
\begin{align}\notag
    P((E \cap H) \cup (F \setminus H);G) &= \hn(\pa^* E \cap H^\one \cap G) + \hn(\partial^* F \cap H^\zero \cap G) \\ \label{eq:cut and paste}
    &\qquad + \hn((E^\one \Delta F^\one)\cap \partial^* H \cap G)\,.
\end{align}

\medskip

Next, following the presentation in \cite[Section 1]{BarCagFus13}, we state a rigidity result from symmetrization. Decomposing $\rn=\{(\overline{x},y):\overline{x} \in \mathbb{R}^{n-1},\,y\in  \mathbb{R}\}$, the {\it one-dimensional slices} of a Borel measurable $E\subset \rn$ (with respect to the subspace $\{(\overline{x},y)\in \rn:y=0 \}$) and their accompanying measures are defined for each $\ovx\in \mathbb{R}^{n-1}$ as
\begin{align}\notag
   E_\ovx:= \{y\in \mathbb{R} : (\ovx,y)\in E \} \quad\textup{and}\quad L_E(\ovx) = \mathcal{H}^1(E_\ovx)\,. 
\end{align}
We also set 
\begin{align}\notag
 \pi(E)^+ := \{\ovx \in \mathbb{R}^{n-1}: L_E(\ovx)>0 \}\,.   
\end{align}
The {\it Steiner symmetral} of $E$ is
\begin{align}\notag
  E^S := \{(\ovx,y) \in \mathbb{R}^n: \ovx\in \pi(E)^+,\,\, |y| < \mathcal{H}^1(E_\ovx)/2 \}\,.  
\end{align}
The classical Steiner inequality says that if $E$ is a set of finite perimeter, then
\begin{align}\label{eq:steiner inequality}
P(E^S ; B \times \mathbb{R}) \leq P(E; B \times \mathbb{R}) \quad \forall \mbox{ Borel $B\subset \mathbb{R}^{n-1}$}\,,
\end{align}
with equality when $B=\mathbb{R}^{n-1}$ implying that $E_{\ovx}$ is an interval for $\mathcal{H}^{n-1}$-a.e.~ $\ovx\in \pi(E)^+$ \cite[Theorem 1.1.(a)]{BarCagFus13}. 

\medskip

To state the theorem on equality cases in \eqref{eq:steiner inequality}, we first introduce the condition
\begin{align}\label{eq:no vertical parts}
    \hn\big(\{x\in \pa^* (E^S) : \nu_{E^S}(x) \cdot e_n = 0 \}\cap (\Omega\times \mathbb{R})\big)=0\,,
\end{align}
for any open $\Omega \subset \mathbb{R}^{n-1}$. Geometrically speaking, \eqref{eq:no vertical parts} says that $\partial^* (E^S)$ has no non-negligible ``vertical parts" over $\Omega$. Under the assumption \eqref{eq:no vertical parts}, the function $L_E$ belongs to $W^{1,1}(\mathbb{R}^{n-1})$ \cite[Proposition 3.5]{BarCagFus13} (as opposed to the weaker $BV(\mathbb{R}^{n-1})$ when \eqref{eq:no vertical parts} fails). Therefore, the Lebesgue average of $L_E$ exists for $\mathcal{H}^{n-2}$-a.e.~ $\ovx\in \mathbb{R}^{n-1}$ \cite[pages 160 and 156]{EvaGar92}. We define the {\it precise representative}
\begin{align}\notag
  L_E^*(\ovx) =  \begin{cases}
  \displaystyle\lim_{r\to 0}\displaystyle\frac{1}{\omega_{n-1}r^{n-1}}\int_{\{|\overline{z}-\ovx|<r\}} L_E(\overline{z})\,d\mathcal{L}^{n-1}(\overline{z}) & \mbox{if the limit exists} \\ 
    0  & \mbox{otherwise} \,,
\end{cases}
\end{align}
so that $L_E^*(\ovx)$ is equal to its Lebesgue average for $\mathcal{H}^{n-2}$-a.e.~ $\ovx\in \mathbb{R}^{n-1}$. The following rigidity result, which we state in co-dimension 1, was proved by Barchiesi-Cagnetti-Fusco in \cite{BarCagFus13} for Steiner symmetrization of arbitrary co-dimension; see also \cite{CagColDePMag14} for additional results on rigidity in \eqref{eq:steiner inequality}.

\begin{theorem}\cite[Theorem 1.2]{BarCagFus13}\label{thm:symm}
   If $\Omega \subset \mathbb{R}^{n-1}$ is a connected open set, $E\subset \mathbb{R}^n$ is a set of finite perimeter such that $P(E^S;\Omega\times \mathbb{R})=P(E;\Omega\times \mathbb{R})$, \eqref{eq:no vertical parts} holds, and $L_E^*(\ovx)>0$ for $\mathcal{H}^{n-2}$-a.e.~ $\ovx\in \Omega$, then $E \cap (\Omega \times \mathbb{R})$ is Lebesgue equivalent to a translation along $\{\overline{0}\}\times \mathbb{R} $ of $E^S \cap (\Omega \times \mathbb{R})$.
\end{theorem}


We will also use a slicing result for clusters adapted from the corresponding result for slices of sets of finite perimeter by lines from \cite[Theorem 2.4]{BarCagFus13}, which in turn is based on Vol'pert \cite{Vol67}. 

\begin{lemma}[Slicing by lines]\label{lemma:slice}
    If $\X$ is an $(N,M)$-cluster in $\rn$, then there exists Borel measurable $F\subset \mathbb{R}^{n-1}$ such that $\hn(\mathbb{R}^{n-1}\setminus F)=0$, and, if $\overline{x}\in F$, then 
\begin{enumerate}[label=(\roman*)]
\item $\X(j)_\ovx\subset \mathbb{R}$ is a set of locally finite perimeter for all $1\leq j \leq N+M$,
\item $(\X(j)^\one)_\ovx \ehone (\X(j)_\ovx)^\oner$ and $(\X(j)^\zero)_\ovx \ehone (\X(j)_\ovx)^\zeror$ for $1\leq j \leq N+M$, 
\item $[\pa^* \X(j) \cap \pa^*\X(k)]_\ovx = \pastr (\X(j)_\ovx) \cap \pastr(\X(k)_\ovx)$ for all $1\leq j < k \leq N+M$, and 
\item $\nu_{\X(j)}(\overline{x},t) \cdot e_n\neq 0$ and $\nu_{\X(j)}(\overline{x},t) \cdot e_n / |\nu_{\X(j)}(\overline{x},t) \cdot e_n| = \nu_{\X(j)_\ovx}(t)$ for every $t\in [\pa^* \X(j)]_\ovx$ and each $1\leq j \leq N+M$.
\end{enumerate}
\end{lemma}

\begin{proof}
Items $(i)$, $(iii)$, and $(iv)$ follow from \cite[Theorem 2.4.$(i)$-$(iii)$]{BarCagFus13} and \eqref{eq:hn partition}, while $(ii)$ is a consequence of Fubini's theorem and the Lebesgue points theorem.
\end{proof}

Finally, we will need a classical result drawn from \cite[Chapter 19]{Mag12} on the liquid drop problem in a halfspace. For a halfspace $H\subset \rn$, set of finite perimeter $E\subset H$, and real number $\beta\in \mathbb{R}$, let
\begin{align}\label{eq:gauss free energy}
    \mathcal{F}_{\beta}(E;H)= P(E;H) - \beta P(E;\partial H)\,.
\end{align}
Before stating the relevant theorem, we motivate it through a lemma connecting \eqref{eq:gauss free energy} to $\pc$ under some symmetry assumptions. 

\begin{lemma}\label{lemma:rewriting cluster energy}
If $\bf c$ is a family of weights satisfying $c_{12}=c_{13}$, $H$ is an open halfspace, $R$ is the reflection map over $\pa H$, and $\X$ is a $(1,2)$-cluster such that $\X(1)\cc B_r(0)$,
    \begin{align}\label{eq:symmmetry assumptions in lemma}
        \X(1) \overset{\mathcal{L}^n}{=} R(\X(1))\,, \quad \X(2) \overset{\mathcal{L}^n}{=} H \setminus \X(1)\,,\quad \mbox{and}\quad\X(3) \overset{\mathcal{L}^n}{=} \rn \setminus (\X(1) \cup \X(2))\,,
    \end{align}
then 
\begin{eqnarray}\label{eq:reduced boundary comps 1}
    &\pa H \shn \X(1)^\one \cup \X(1)^\zero\,,\quad \pa H \cap  \pa^*(\X(1) \cap H) \ehn \X(1)^\one \cap \pa H\,,& \\ \label{eq:reduced boundary comps 2}
    & \pa H \cap (\X(1) \cap H)^\zero= \pa H \cap \X(1)^\zero \ehn \X(2,3)\,,&
\end{eqnarray}
and
\begin{align}\label{eq:energy of symmetric cluster}
    \pc(\X;C_r) = 2c_{13}\mathcal{F}_{c_{23}/(2c_{13})}(\X(1) \cap H;H) +  c_{23}\omega_{n-1}r^{n-1} \,.
\end{align}
\end{lemma}

\begin{proof}
First, by Federer's theorem \eqref{eq:fed thm} and the fact that the Lebesgue density of $\X(1) \cap H$ along $\pa H$ is at most $1/2$,
\begin{align}\label{eq:hyperplane breakdown}
    \pa H \ehn \pa H \cap \big[\pa^*(\X(1) \cap H) \cup (\X(1) \cap H)^\zero \big]\,.
\end{align}
Through routine manipulations which we omit based on \eqref{eq:hyperplane breakdown}, \eqref{eq:fed thm}, and \eqref{eq:cut and past intersection}, one concludes \eqref{eq:reduced boundary comps 1}-\eqref{eq:reduced boundary comps 2}. For \eqref{eq:energy of symmetric cluster}, we start by using the equality $c_{12}=c_{13}$, \eqref{eq:hn partition}, and $\X(1) \cc B_r(0) \subset C_r$ to rewrite
\begin{align}\label{eq:0th term}
   \pc(\X;C_r) &= c_{13}P(\X(1)) + c_{23}\hn(\X(2,3) \cap C_r)\,.
\end{align}
By the first equivalence in \eqref{eq:reduced boundary comps 1} and the fact that $\X(1)$ is symmetric over $\pa H$,
\begin{align}\label{eq:1st term}
    c_{13}P(\X(1)) = c_{13}P(\X(1);\rn  \setminus \pa H) =  2c_{13}P(\X(1); H) = 2c_{13}P(\X(1) \cap H ; H)\,.
\end{align}
Also, from \eqref{eq:reduced boundary comps 1}-\eqref{eq:reduced boundary comps 2},
\begin{align}\notag
   c_{23}\hn(\X(2,3) \cap C_r) &= c_{23}\hn( \pa H \cap (\X(1) \cap H)^\zero \cap C_r) = c_{23}\hn(\pa H \cap \X(1)^\zero \cap C_r) \\ \label{eq:2nd term}
   &= c_{23}\big[\hn(\pa H \cap C_r) - \hn(\pa H \cap \pa^*(\X(1) \cap H) \cap C_r) \big]\,.
\end{align}
The equality \eqref{eq:energy of symmetric cluster} follows by inserting the sum of \eqref{eq:1st term} and \eqref{eq:2nd term} into \eqref{eq:0th term}.
\end{proof}

\noindent The next theorem is a rephrasing of \cite[Theorem 19.21]{Mag12} and characterizes minimizers for the liquid drop problem without gravity in a halfspace in terms of standard weighted lens clusters.

\begin{theorem}\label{thm:liquid drops}
    If $\bf c$ is a family of weights satisfying \eqref{eq:pos}-\eqref{eq:triangle} and $c_{12}=c_{13}$, $\xl$ is the standard weighted lens cluster corresponding to the weights $\bf c$, and $H=\{x\in \rn:x_n>0\}$, then $\xl(1)\cap H$ is the unique minimizer up to horizontal translations and $\mathcal{L}^n$-null sets for the variational problem
\begin{align}\label{eq:liquid drop problem anchoring}
   &\inf \big\{\mathcal{F}_{c_{23}/(2c_{13})}(E;H): |E|=1/2,\, E \subset H,\, P(E)<\infty \big\}\,. 
\end{align}
\end{theorem}

\begin{proof}
    First, the assumptions \eqref{eq:pos}-\eqref{eq:triangle} and $c_{12}=c_{13}$ imply that
\begin{align}\notag
  \beta := c_{23}/(2c_{13}) = c_{23}/(c_{12}+c_{13})\in (0,1)\,.  
\end{align}
    The only assumption needed to apply the characterization of minimizers in \eqref{eq:liquid drop problem anchoring} from \cite[Theorem 19.21]{Mag12} is $|\beta|<1$, and thus the unique minimizer up to horizontal translations and Lebesgue null sets for \eqref{eq:liquid drop problem anchoring} is $B \cap H$, where $B$ is the ball determined by the two conditions $|B \cap H| = 1/2$ and
    \begin{align}\label{eq:young's law}
    -e_n \cdot \nu_B(x) =-c_{23}/(2c_{13}) \qquad \forall x\in \partial B \cap \partial H\,.
    \end{align}
    The angle $\theta$ formed by $\nu_B$ and $-e_n$ along $\partial H$ is a $90^\circ$ counterclockwise rotation of the angle $\theta_2$ formed by $\partial (H \setminus B)$ along its singular set $\partial B \cap \partial H$. So the cluster $\X$ given by
\begin{align}\notag
\X(1) = (B \cap H) \cup R(B \cap H)\,, \quad \X(2) = H \setminus B\,, \quad \X(3) = H^c \setminus R(B \cap H)     
\end{align}
is the standard symmetric lens cluster corresponding to the angles $\theta_2=\theta_3$ along the singular sets of $\pa \X(2)$ and $\pa \X(3)$ and $\theta_1 = 2\pi - 2\theta_2$ formed by the two spherical caps comprising $\partial \X(1)$. By multiplying \eqref{eq:young's law} by $-2\sin \theta_2/c_{23}$ on both sides, we obtain
\begin{align}\notag
   \frac{-2\cos \theta_2 \sin \theta_2}{c_{23}} = \frac{\sin \theta_2}{c_{13}}\,.
\end{align}
Since $\theta_2 = \pi - \theta_1/2$, the left hand side simplifies to $\sin \theta_1/c_{23}$. By $c_{12}=c_{13}$ and $\theta_2=\theta_3$, the resulting equations for $\theta_i$ define the standard weighted lens cluster corresponding to the family $\bf c$, so $|\X(j)\Delta\xl(j)|=0$ for $1\leq j \leq 3$. Thus $\xl (1)\cap H$ is Lebesgue equivalent to $B \cap H$, and the theorem is complete.
\end{proof}

We end the preliminaries with a short discussion of weighted double bubbles \cite{Law14}. 

\begin{definition}[Standard weighted double bubble]\label{def:standard weighted double bubble}
    The {\it standard weighted double bubble} $\xbub$ in $\rn$ associated to weights ${\bf c}$ satisfying \eqref{eq:pos}-\eqref{eq:triangle} and $m>0$ is the $(2,1)$-cluster such that
\begin{enumerate}[label=(\roman*)]
    \item $|\xbub(1)|=1$, $|\xbub(2)|=m$,
    \item $\pa \xbub(j,k)$ are three spherical caps, meeting along an $(n-2)$-dimensional sphere $\pa D$ contained in the plane $\{x_n=0\}$, with $\pa \xbub(2) \cap \pa \xbub(3) \subset \{x_n\geq 0\}$, and
    \item for each $j$, the two smooth surfaces forming $\partial \xbub(j)$ and meeting along $\pa D$ form an angle $\theta_j\in (0,\pi)$, and 
\begin{align}\notag
\frac{\sin \theta_1}{c_{23}}=\frac{\sin \theta_2}{c_{13}}=\frac{\sin\theta_3}{c_{12}}\,,    
\end{align}
where $\theta_1+\theta_2+\theta_3=2\pi$. \end{enumerate}
\end{definition}

\noindent In \cite{Law14}, Lawlor has proved that 
\begin{align}\label{eq:minimality of standard bubble}
    \pc(\xbub) \leq \pc (\X)
\end{align}
for any $(2,1)$-cluster $\X$ with volume vector $(1,m,\infty)$, and that equality holds if and only if $\X=\xbub$ up to Lebesgue null sets and rigid motions of $\rn$. The last lemma, for which we omit the proof, follows from the definitions of $\xbub$ and $\xl$.

\begin{lemma}[Convergence of weighted double bubbles to weighted lens clusters]\label{lemma:convergence of weighted double bubbles}
    If $\bf c$ satisfies \eqref{eq:pos}-\eqref{eq:triangle}, then as $m\to \infty$, $\xbub(j)$ converges locally in the Hausdorff distance to $\xl(j)$ for each $1\leq j \leq 3$, and for any $R$ with $\xl(1) \cc B_R(0)$, $\pc(\xbub;B_R(0)) \to \pc (\xl;B_R(0))$.
\end{lemma}

\section{Properties of \texorpdfstring{$\bf c$}{c}-locally minimizing \texorpdfstring{$(N,M)$}{NM}-clusters}\label{subsec:prelim}

We present the relevant properties of the weighted energy $\pc$ and $\bf c$-locally minimizing $(N,M)$-clusters. First, by a result of Ambrosio and Braides \cite[Example 2.8]{AmbBra90} (see also \cite[Section 7]{Whi96}), \eqref{eq:pos}-\eqref{eq:triangle} imply that $\pc$ is lower-semicontinuous, that is, if $\X_m\overset{\loc}{\to}\X$ and $U\subset \rn$ is open, then
\begin{align}\label{eq:lsc}
    \pc(\X;U) \leq \liminf_{m\to \infty} \pc (\X_m;U)\,.
\end{align}
The lower-semicontinuity still holds even if some chambers of $\X_m$ vanish in the local limit, so that the limiting object $\X$ is an $(N',M')$-cluster where $N\leq N'$, $M\leq M'$ and the energy on the left hand side is $\pcprime(\X;U)$, where $\bf c'$ is the subfamily of indices corresponding to non-vanishing chambers. 

\begin{lemma}[Volume density estimate]\label{lemma:density estimates}
If $\bf c$ is a family of weights satisfying \eqref{eq:pos}-\eqref{eq:triangle} and $\X$ is a $\bf c$-locally minimizing $(1,2)$-cluster, then there exists $\Lambda>0$, $\e\in (0,1)$, and $r_0>0$ such that
\begin{align}\label{eq:lambda minimality}
    \pc(\X;B_{r_0}(x)) \leq \pc(\X';B_{r_0}(x)) + \Lambda|\X(1) \Delta \X'(1)|
\end{align}
whenever $\X(j) \Delta \X'(j)\subset B_{r_0}(x)$ for some $x\in \rn$ and each $j$, and
\begin{align}\label{eq:lower volume bound}
    |B_r(y) \cap \X(j)| \geq \e\omega_n r^n \quad \mbox{if $y\in \pa \X(j)$, $r<r_0$, and $1\leq j \leq 3$}\,. 
\end{align}
As a consequence, each $\X(j)^\one$ is open and satisfies $\cl \pa^* \X(j) = \pa \X(j)^\one$, and $\X(1)$ is bounded.
\end{lemma}

\begin{proof}
By the volume fixing variations construction for sets of finite perimeter (e.g. \cite[Lemma 17.21]{Mag12}), we may choose $z_1,z_2\in\pa^* \X(1)$, disjoint balls $B_{|z_1-z_2|/4}(z_\ell)$ for $\ell=1,2$, $\sigma_0>0$, $C>0$, and a pair of one parameter families of diffeomorphisms $\{f_t^\ell\}_t$, $\ell=1,2$, such that for every $\s\in [-\s_0,\s_0]$, there exists $t \in [-C\s, C\s]$ such that for each $1\leq j \leq 3$, $f_t^\ell(\X(j)) \Delta \X(j) \subset\!\subset B_{|z_1-z_2|/4}(z_\ell)$,
    \begin{align*}
        |f_t^\ell(\X(1))|=|\X(1)|+\s\,,\quad\mbox{and}\quad |P(f_t^\ell(\X(j));B_{|z_1-z_2|/4}(z_\ell)) - P(\X(j);B_{|z_1-z_2|/4}(z_\ell))| \leq C|\s|\,.
    \end{align*}
If we let $r_0=\min\{|z_1-z_2|/4,(\sigma_0/\omega_n)^{1/n}\}$, then any $B_{r_0}(x)\subset \rn$ is disjoint from $B_{|z_1-z_2|/4}(z_{\ell_0})$ for some $\ell_0\in\{1,2\}$. Therefore, given any $\X'$ and $B_{r_0}(x)$ as in the statement of the lemma, we can modify $\X'$ on $B_{|z_1-z_2|/4}(z_{\ell_0})$ by replacing each $\X'(j) \cap B_{|z_1-z_2|/4}(z_{\ell_0})$ with $f_t^{\ell_0}(\X(j)) \cap B_{|z_1-z_2|/4}(z_{\ell_0})$, where $t$ is chosen based on $\sigma = |\X(1)| - |\X'(1)|$.  We then test the minimality of $\X$ against this modified $\X'$, which is admissible by our choice of $\sigma$, and \eqref{eq:lambda minimality} is verified with $\Lambda=\Lambda(C,{\bf c})$.

\medskip

Moving on to \eqref{eq:lower volume bound}, fix any $r<r_0$ and $y\in \partial \X(1)$; the cases $y\in \partial \X(j)$ where $j=2,3$ are proved similarly. Let $m(r)=|\X(1) \cap B_r(y)|$. We claim that we may decrease $r_0$ and choose $C_1>0$, both independently of $y$, such that, for every $r<r_0$ with
\begin{align}\label{eq:density r assumption 1}
&\hn (\pa^* \X(j) \cap \pa B_r(y))=0\quad\forall 1 \leq j \leq 3\quad \textup{and} \\ \label{eq:density r assumption 2}
&m'(r) = \hn(\X(1)^\one \cap \pa B_r(y))\,,
\end{align}
we have
\begin{align}\label{eq:diff ineq}
    m(r)^{(n-1)/n}\leq C_1 m'(r)\,.
\end{align}
Note that \eqref{eq:density r assumption 1} and \eqref{eq:density r assumption 2} both hold at almost every $r<r_0$ by the local finiteness of $\hn \mres \pa^* \X(j)$ and differentiating $m$ using the co-area formula, respectively.
The derivation of \eqref{eq:lower volume bound} from \eqref{eq:diff ineq} follows dividing by $m(r)^{(n-1)/n}$, which is valid since $y\in \pa \X(1) = \cl \pa^* \X(1)$ implies that $m(r)>0$ for every $r>0$, and integrating. To prove \eqref{eq:diff ineq} for suitable $r<r_0$, let us suppose
\begin{align}\label{eq:smaller choice}
  \hn(\X(1,2) \cap B_r(y))\leq \hn(\X(1,3) \cap B_r(y))\,;  
\end{align}
the case with the opposite inequality is the same. For the subsequent calculation it is convenient to use the rewritten formula for the energy from \eqref{eq:pc rewrite}. Testing \eqref{eq:lambda minimality} on $B_{r_0}(y)$ with $\X'$ where $\X'(1) = \X(1) \setminus B_r(y)$, $\X'(2)=\X(2)$, and $\X'(3) = \X(3) \cup (B_r(y) \cap \X(1))$ and cancelling like terms, we obtain 
\begin{align}\notag
    (c_1+c_2)&\hn(\X(1,2) \cap B_r(y)) + (c_1+c_3)\hn(\X(1,3) \cap B_r(y))\\ \notag
    &\quad \leq (c_2+c_3)\hn( \X(1,2) \cap B_r(y)) + (c_1+c_3)\hn(\X(1)^\one \cap \pa B_r(y)) + \Lambda m(r)\,,
\end{align}
where we used \eqref{eq:density r assumption 1} and \eqref{eq:cut and paste} to compute $\pc(\X';\pa B_r(y))$. After adding $c_1\hn(\X(1)^\one \cap \pa B_r(y))=c_1m'(r)$ to both sides, we simplify using \eqref{eq:smaller choice} and find
\begin{align}
    c_1 P(\X(1)\cap B_r(y)) = c_1P(\X(1);B_r(y)) + c_1 m'(r) \leq (2c_1+c_3)m'(r) + \Lambda m(r)\,.
\end{align}
Applying the isoperimetric inequality to $\X(1) \cap B_r(y)$ then gives
\begin{align}\notag
  c_1 n\,\omega_n^{1/n}m(r)^{(n-1)/n} \leq (2c_1+\max\{c_2,c_3\}) m'(r) + \Lambda m(r)\,.
\end{align}
Decreasing $r_0$ if necessary based on $\Lambda$ so that we may absorb $\Lambda m(r)$ onto the left hand side, \eqref{eq:diff ineq} follows. To see that each $\X(j)^\one$ is open, if $y\in \X(j)^\one$, then  $|B_{\rho}(y) \cap \X(j)| >(1-\e/2^{n+1}) \omega_n\rho^n$ for some $\rho<r_0$. Then for any $z\in B_{\rho/2}(y)$ and $k\neq j$,
\begin{align}\label{eq:z not in the ball}
|B_{\rho/2}(z) \cap \X(k) | \leq |B_{\rho}(y) \cap \X(k)| \leq \omega_n \rho^n - |B_{\rho}(y) \cap \X(j)| < \e\omega_n\rho^n/2^{n+1}\,.    
\end{align}
By the density estimate \eqref{eq:lower volume bound} on $B_{\rho/2}(z)$ when $z\in \pa \X(k)$, \eqref{eq:z not in the ball} prevents $z$ from belonging to $\pa \X(k)$ for any $k\neq j$. Since $z\in B_{\rho/2}(y)$ was arbitrary, $\pa \X(k) \cap B_{\rho/2}(y) = \varnothing$ for each $k\neq j$. Then by $y\in \X(j)^\one$ (which precludes $B_\rho(y) \subset \X(k)$), this implies that
\begin{align*}
    |B_{\rho /2}(y) \cap \X(k) | = 0\quad \mbox{for }k\neq j\,.
\end{align*}
Thus $B_{\rho/2}(y) \subset \X(j)^\one$, and $\X(j)^\one$ is open. For the characterization of $\cl \pa^* \X(j)$, first note that by \eqref{eq:fed thm} and the fact that $x\in \X(j)^\half$ entails $B_r(x) \cap \X(j)^t\neq \varnothing$ for all $r>0$ and $t\in \{0,1\}$, $\cl \pa^* \X(j) \subset \cl \X(j)^\half\subset \pa \X(j)^\one$. The reverse inclusion follows from the equivalence $\cl \pa^* \X(j) = \{x:0<|B_r(x)\cap \X(j)| < \omega_nr^n\,\forall r\}$ \cite[Eq.~ (12.12)]{Mag12}. Lastly, \eqref{eq:lower volume bound} and $|\X(1)|<\infty$ imply that $\pa \X(1)$ is bounded, which together with $|\X(1)|<\infty$ implies that $\X(1)$ is bounded also.  \end{proof}

\begin{remark}[Bounded proper chambers for $\bf c$-locally minimizing $(N,2)$- and $(N,3-N)$-clusters]\label{remark:bounded proper chambers}
The final conclusion of Lemma \ref{lemma:density estimates}, namely the boundedness of the proper chambers, holds for general $\bf c$-locally minimizing $(N,2)$-clusters by a different argument. First, following Almgren \cite[VI.10-12]{Alm76}, one obtains an almost-minimality inequality 
\begin{align}\label{eq:lambda min ineq}
    \pc(\X;B_r(x)) \leq \pc(\X';B_r(x)) + \sum_{j}\Lambda |\X(j) \Delta \X'(j)|
\end{align}
as in \cite[Theorem 5.1, Lemma 5.2, and Remark 5.3]{NovPaoTor23} in the paper of Novaga-Paolini-Tortorelli. Second, \eqref{eq:lambda min ineq} allows one to run the elimination argument of Leonardi \cite[Theorem 3.1]{Leo01} on any subfamily $\mathcal{I} \subset \{1,\dots, N+M\}$ such that $\# \mathcal{I} = N+M-2$, yielding positive $\eta$ and $\rho$ such that
\begin{align}\label{eq:elimination us}
    \mbox{``if $r<\rho$ and $|B_r(x) \cap (\cup_{j\in \mathcal{I}} \X(j))|\leq \eta r^n$, then $(\cup_{j\in \mathcal{I}} \X(j))\cap B_{r/2}(x)=\varnothing$"}\,.
\end{align}
The boundedness of the proper chambers $\cup_{j=1}^N \X(j)$ when $M=2$ follows from applying \eqref{eq:elimination us} to $\mathcal{I}=\{1,\dots,N\}$, although does not appear to be an immediate consequence when $M>2$ and $N>1$. When $N+M=3$, \eqref{eq:elimination us} applied to $\X(j)$ for each $j$ can be used to deduce lower volume density bounds on every chamber and thus boundedness of proper ones. Although this covers the case of Lemma \ref{lemma:density estimates}, we have included that proof since it is simpler than \eqref{eq:elimination us}. To our knowledge, lower volume density bounds as in \eqref{eq:lower volume bound} are not known for locally minimizing clusters with general weights $\bf c$ satisfying \eqref{eq:pos}-\eqref{eq:triangle} when there are four or more chambers. In the equal weights case, lower density bounds/boundedness of the proper chambers were proved by Novaga-Paolini-Tortorelli \cite[Theorem 2.4]{NovPaoTor23}.
\end{remark}

\begin{remark}[Further properties of $\bf c$-locally minimizing $(N,M)$-clusters]
The property \eqref{eq:elimination us} (see also \cite[Property P]{Whi96}) implies that $\X(j)^\one$ is open for each $j$ and, combined with a first variation argument, that each $\X(j,k)$ is an analytic hypersurface with constant mean curvature $\lambda_{jk}$, which is $0$ when $\min\{j,k\}\geq N+1$. In addition, by an announced result of White \cite{Whi86,Whi96}, under the assumptions \eqref{eq:pos}-\eqref{eq:triangle}, the singular set of the interfaces should have Hausdorff dimension at most $n-2$; see also \cite{ColEdeSpo22}. We also refer the reader to \cite{LawMor94,Cha95,Mor98} and the references therein for more on immiscible fluid clusters.
\end{remark}

\begin{theorem}[Exterior monotonicity formula for $\bf c$-locally minimizing $(N,M)$-clusters]
If $\X$ is a $\bf c$-locally minimizing $(N,M)$-cluster and $R>0$ is such that $\cup_h \E(h) \cc B_R(0)$, then 
\begin{align}\label{eq:ext stationary}
  0=  \sum_{N+1\leq j < k \leq N+M}\int_{\X(j,k) \setminus \cl B_R(0)} c_{jk}\,\mathrm{div}^{\X(j,k)}\,X(x)\, \,d\hn(x)\quad \forall X\in C_c^\infty(\rn \setminus \cl B_R(0);\rn)\,.
\end{align}
If $R$ is also a point of differentiability of the increasing function $r\mapsto \pc(\X;B_r(0))$, which includes almost every $r>0$, then there exists $C(R)\in \mathbb{R}$ such that
\begin{align}\label{eq:ext mon formula}
     \frac{\pc(\X ; B_s(0) \setminus B_R(0))}{s^{n-1}} +\frac{C(R)}{s^{n-1}}\leq \frac{\pc(\X;B_r(0) \setminus B_R(0))}{r^{n-1}} + \frac{C(R)}{r^{n-1}}\qquad \forall R<s<r<\infty\,.
\end{align}
\end{theorem}

\begin{proof}
It is convenient to use the language of varifolds (see e.g. \cite{Sim83}), so we introduce the locally $\hn$-rectifiable set
\begin{align*}
S=\cup_{N+1\leq j < k \leq N+M }\X(j,k)\setminus \cl B_R(0) = \cup_{1\leq i < i' \leq M} \big[\pa^* \F(i) \cap \pa^* \F(i')\big] \setminus \cl B_R(0)\,,
\end{align*}
Borel measurable multiplicity function
\begin{align}
    \theta = \sum_{N+1\leq j < k \leq N+M} c_{jk} \mathbf{1}_{\X(j,k)\setminus \cl B_R(0)}\,,
\end{align}
and the rectifiable varifold $V=\var(S,\theta)$ induced by $S$ and $\theta$. The assumption that $\cup_h \E(h)\cc B_R(0)$ implies that for any $X\in C_c^\infty(\rn \setminus \cl B_R(0);\rn)$, every cluster in the one parameter family of $(N,M)$-clusters induced by $\X$ and the one parameter family of diffeomorphisms with initial velocity $X$ has volume vector equal to $\m(\X)$. Testing the $\bf c$-local minimality of $\X$ against this family gives \eqref{eq:ext stationary}. In terms of $V$, this shows that $V$ is stationary in $\rn \setminus \cl B_R(0)$. Denoting by $\|V\|$ the weight of $V$, we have $\|V\|(B_R(0))=0$ by construction, and thus
\begin{align}\label{eq:ext stationary proof}
    \delta V(X) = \int \mathrm{div}^S X(x)\,d\|V\|(x)=0 \qquad \forall X \in C_c^\infty(\rn \setminus \pa B_R(0);\rn)\,.
\end{align}

\medskip

Turning now towards \eqref{eq:ext mon formula}, we first claim that there exists a Radon measure $\sigma$ supported on $\pa B_R(0)$ and Borel measurable $\nu^{\rm co}_V:\pa B_R(0) \to \rn$ with $|\nu^{\rm co}_V|=1$ $\sigma$-a.e.~ on $\pa B_R(0)$ such that
\begin{align}\label{eq:div thm}
    \delta V(X) = \int X(x) \cdot \nu^{\rm co}_V(x)\,d\sigma(x)\qquad \forall X \in C_c^\infty(\rn;\rn)\,.
\end{align}
To prove \eqref{eq:div thm}, let $\{\varphi_\e\}_{\e>0}\subset C_c^\infty(\rn;[0,1])$ be a family of cut-off functions such that $\varphi_\e = 1$ on $\pa B_R(0)$, $\spt\, \varphi_\e \subset  B_{R+2\e}(0)\setminus \cl B_{R-2\e}(0)$, and $\|\nabla \varphi \|_{L^\infty}\leq \e^{-1}$. Then by \eqref{eq:ext stationary proof} and our differentiability assumption on $R$,
\begin{align}\notag
    \delta V(X) &= \lim_{\e \to 0} \int \mathrm{div}^S\, \big[X(1-\varphi_\e)\big](x) \,d\|V\|(x) + \int \mathrm{div}^S\, \big[X\varphi_\e\big](x) \,d\|V\|(x) \\ \notag
    &\leq  \limsup_{\e \to 0} \big[\e^{-1}\|X\|_{L^\infty} + \|\mathrm{div}^S\, X\|_{L^\infty}\big]\|V\|(B_{R+2\e}(0)\setminus \cl B_{R-2\e}(0))\\ \label{eq:bounded first var}
    &\leq C(R,\X)\|X\|_{L^\infty}\,.
\end{align}
Thus $V$ has bounded first variation in $\rn$, and so by appealing to the Riesz representation theorem, we obtain a $\rn$-valued Radon measure $\mu$ such that
\begin{align}\notag
 \delta V(X) = \int X(x) \cdot d\mu(x)\qquad \forall X \in C_c^\infty(\rn;\rn)\,.   
\end{align}
But by \eqref{eq:ext stationary proof}, $|\mu|(\rn \setminus \pa B_R(0))=0$, which implies that, setting $\sigma=|\mu|$, $\s$ is supported on $\pa B_R(0)$ and there exists $\nu^{\rm co}_V:\pa B_R(0) \to \rn$ with $|\nu^{\rm co}_V|=1$ $\sigma$-a.e.~ such that \eqref{eq:div thm} holds. The monotonicity formula derived from \eqref{eq:div thm} (following the computations in \cite[Section 17]{Sim83}) with
\begin{equation}\notag
C(R) = \frac{-1}{(n-1)}\int x\cdot \nu^{\rm co}_V(x)\,d\sigma(x)    
\end{equation}
can be found in \cite[Theorem 2.7.$(i)$]{MagNov22}; see also \cite[Equation 1.8]{Sim85}.
\end{proof}

\noindent As an immediate corollary of the previous two results, we can justify Definition \ref{def:growth def}.

\begin{corollary}[Density at infinity]\label{corollary:density at infinity}
    If $\X$ is a $\bf c$-locally minimizing $(N,M)$-cluster and $\cup_h \E(h)\subset B_R(0)$ for some $R>0$, as is the case for $N=1$ and $M=2$, then
\begin{align}\label{eq:dens at inf in corollary}
   \Theta_\infty(\X)= \lim_{r\to \infty}\frac{\pc(\X;B_r(0))}{r^{n-1}}
\end{align}
exists and equals $\lim_{r\to \infty}r^{1-n}\pc(\X;B_r(0)\setminus \cl B_R(0))$ for every $R>0$.
\end{corollary}

\noindent The next corollary establishes the existence up to subsequences of blow-down clusters.

\begin{corollary}[Existence of blow-down clusters]\label{corollary:tangent cones at infinity}
If $\bf c$ satisfies \eqref{eq:pos}-\eqref{eq:triangle}, $\X$ is a $\bf c$-locally minimizing $(N,M)$-cluster such that $\cup_h \E(h)\subset B_R(0)$ for some $R>0$, and $R_m\to \infty$ then, up to a subsequence, there exists $M' \leq M$, subfamily $\mathcal{I}:=\{j_1,\dots, j_{M'}\}\subset \{N+1,\dots, N+M\}$ and a conical $(0,M')$-cluster $\X_\infty$ such that, setting ${\bf c}'=\{c_{jk}\}_{j,k\in \mathcal{I}}$, $\X_\infty$ is ${\bf c}'$-locally minimizing,
\begin{align}\label{eq:L1 conv of blow-downs}
    &\X(j_\ell)/R_m \overset{\loc}{\to} \X_\infty(\ell)\quad \forall 1\leq \ell \leq M'\,, \qquad \mbox{and}\\ \label{eq:density of tangent cone}
    &\pcprime(\X_\infty;B_1(0))= \Theta_\infty(\X)\,.
\end{align}
\end{corollary}

\begin{proof}
    By Corollary \ref{corollary:density at infinity}, $\sup_m \pc(\X/R_m;B_R(0)) < \infty$ for every $R>0$. Therefore, the existence of a (not relabelled) subsequence of radii $R_m\to \infty$, $M' \leq M$, subfamily $\{j_1,\dots,j_{M'}\}\subset \{N+1,\dots,N+M\}$, and $(0,M')$-cluster $\X_\infty$ such that \eqref{eq:L1 conv of blow-downs} holds is a consequence of \eqref{eq:pos} and compactness in $BV$. Furthermore, by \eqref{eq:lsc} and \eqref{eq:dens at inf in corollary}, we have
\begin{align}\label{eq:upper bound on density}
    \pcprime(\X_\infty;B_R(0)) \leq \liminf_{m\to \infty}\pc(\X/R_m;B_R(0)) =\Theta_\infty(\X) R^{n-1}\qquad \forall R>0\,.
\end{align}
We claim now that for any $(0,M')$-cluster $\X'$ with $\X'(\ell) \Delta \X_\infty(\ell) \cc B_R(0)$ for some $R>0$ and each $\ell$,
\begin{align}\label{eq:min ineq in tangent proof}
    \pcprime(\X_\infty;B_R(0)) \leq \pcprime(\X';B_R(0))\,.
\end{align}
To prove \eqref{eq:min ineq in tangent proof}, choose $R'<R$ such that $\X'(\ell)\Delta \X_\infty(\ell) \cc B_{R'}(0)$ for each $\ell$ and 
\begin{equation}
\begin{aligned}\label{eq:slice difference vanishes}
&\hn\big((\X(j)^\one / R_m) \cap \pa B_{R'}(0) \big)\to 0\quad \forall j\notin\mathcal{I}\,,\quad\mbox{and}  \\ 
&\hn\big(((\X(j_{\ell})^\one/R_m)\Delta \X_\infty(\ell)^\one) \cap \partial B_{R'}(0)\big) \to 0\quad \mbox{for all $1\leq \ell \leq M'$}\,; 
\end{aligned}
\end{equation}
this is possible by the co-area formula and \eqref{eq:L1 conv of blow-downs}. Let $\X_m'$ be the $(N,M)$-clusters defined by
\begin{align}\notag
    \X_m'(j)&= \X(j) \quad \mbox{if $1\leq j \leq N$}\,, \\ \notag
    \X_m'(j_\ell) &= \big[R_m(\X'(\ell)\cap B_{R'}(0))\cup (\X(j_\ell)\setminus B_{R_mR'}(0))\big]\setminus \cup_{1\leq j\leq N} \X(j)\quad \mbox{if $1\leq \ell \leq M'$}\,, \\ \notag
    \X'_m(j) &= \X(j) \setminus B_{R_mR'}(0) \quad\mbox{if $N+1\leq j \leq N+M$, $j\notin \mathcal{I}$}\,.
\end{align}
Then by \eqref{eq:lsc}, \eqref{eq:min condition} applied to the $\bf c$-local minimizers $\X/R_m$, and \eqref{eq:cut and paste},
\begin{align}\notag
    \pcprime(\X_\infty;B_{R'}(0)) &\leq \liminf_{m\to \infty} \pc (\X/R_m;B_{R'}(0))
    \leq \liminf_{m\to \infty} \pc (\X'_m/R_m; B_{R'}(0))\\ \notag
    &\leq \liminf_{m\to \infty}\bigg[\sup_{j<k} c_{jk}\sum_{1\leq j\leq N} P(\X(j))/R_m^{n-1}\bigg] + \pcprime (\X';B_{R'}(0)) \\ \notag
    &\qquad + \liminf_{m\to \infty}\sup_{j<k} c_{jk} \sum_{j\in \{N+1,\dots,N+M\}\setminus \mathcal{I}} \hn\big((\X(j)^\one / R_m) \cap \pa B_{R'}(0) \big) \\ \notag
    &\qquad + \liminf_{m\to \infty}\sup_{j<k} c_{jk} \sum_{1\leq \ell \leq M'}\hn\big(((\X(j_{\ell})^\one/R_m)\Delta \X_\infty(\ell)^\one) \cap \partial B_{R'}(0)\big)\,.
\end{align}
By the fact that $\sum_{j=1}^N P(\X(j))$ is finite and \eqref{eq:slice difference vanishes}, \eqref{eq:min ineq in tangent proof} follows. 

\medskip

We now choose representatives for each $\X_\infty(\ell)$ such that $\cl \partial^* \X_\infty(\ell) = \partial \X_\infty(\ell)$ (see e.g. \cite[Remark 15.3]{Mag12}). Since $\X_\infty$ also satisfies the minimality inequality \eqref{eq:min ineq in tangent proof}, $\X_\infty$ is a $\bf c'$-locally minimizing $(0,M')$-cluster. It remains to prove \eqref{eq:density of tangent cone} and that $\X_\infty$ is conical (each chamber is a cone). First, we claim that
\begin{align}\label{eq:same density for all r}
 \pcprime(\X_\infty;B_R(0)) = \Theta_\infty(\X) R^{n-1}\qquad \mbox{for a.e.~ $R>0$}\,.  
\end{align}
By \eqref{eq:upper bound on density}, we know that $\pcprime(\X_\infty;B_R(0)) \leq \Theta_\infty(\X) R^{n-1}$ for all $R>0$. The reverse inequality is obtained by choosing $R$ such that \eqref{eq:slice difference vanishes} holds at $R$ (which is a.e.~ $R>0$) and recalling that
\begin{align}\notag
    \Theta_\infty(\X) R^{n-1}= \lim_{m\to \infty} \pc(\X;B_{R_mR}(0))/R_m^{n-1}\,.
\end{align}
We then test the minimality of $\X/R_m$ on $B_{R}(0)$ against the competitor which agrees with $\X_\infty$ inside $B_{R}(0)$ and $\X/R_m$ outside. By a similar argument as in the preceding paragraph, this yields
\begin{align}\notag
 \Theta_\infty(\X)R^{n-1} \leq   \pcprime(\X_\infty;B_R(0))\quad \mbox{for a.e.~ $R>0$}\,,
\end{align}
which finishes the proof of \eqref{eq:same density for all r}. Now $\X_\infty$ is a $\bf c'$-locally minimizing $(0,M')$-cluster on all of space, and so the varifold 
\begin{align}\notag
    V = \var \Big(\bigcup_{1\leq \ell < \ell' \leq M'}\X_\infty(\ell,\ell'), \sum_{1\leq \ell<\ell'\leq M'}c_{{j_\ell}{j_{\ell'}}}\mathbf{1}_{\X_\infty(\ell,\ell')} \Big)
\end{align}
is stationary in all of space. Since $V$ has constant energy density by \eqref{eq:same density for all r} and is stationary, it is a cone, and thus so is $\X_\infty$.
\end{proof}

\begin{remark}[Uniqueness of the blow-down cluster]
    The uniqueness of blow-downs of a $\bf c$-locally minimizing cluster $\X$ is not known in full generality, although it should hold in some specific cases. For example, if a single blow-down has two non-empty chambers and interface which has an isolated singularity at the origin, then \cite[Theorem 1.1]{MinEde24} or versions of the blow-up arguments in \cite{AllAlm81,Sim83a,EngSpoVel19} adapted to blow-downs of exterior minimal surfaces (see also \cite[pg 269]{Sim85a}) yield uniqueness. The next corollary proves uniqueness and decay when some blow-down is a pair of halfspaces.
\end{remark}

\begin{corollary}[Asymptotic expansion for $\bf c$-locally minimizing $(N,2)$-clusters with planar growth]\label{corollary:asymptotic expansion}
    If $\X$ is a $\bf c$-locally minimizing $(N,2)$-cluster in $\rn$ with bounded proper chambers, and in addition $\Theta_\infty(\X) = c_{(N+1)(N+2)}\omega_{n-1}$ if $n\geq 8$, then, up to a rotation, there exists $R_0>0$ such that on $\rn \setminus C_{R_0}$, $\partial \X\setminus C_{R_0}=\X(N+1,N+2) \setminus C_{R_0}$ coincides with the graph of a solution $u:\mathbb{R}^{n-1}\setminus B_{R_0}^{n-1}(\overline{x})$ to the minimal surface equation satisfying the expansion at infinity
\begin{equation}\label{eq:asymptotic exp}
  u(\ovx) = \begin{cases}
  a & \overline{x}\in \mathbb{R},\,\, n=2\,,\\ 
    a + \displaystyle\frac{b}{|\overline{x}|^{n-3}}+ \displaystyle\frac{\overline{c}\cdot \overline{x}}{|\overline{x}|^{n-1}}+ \mathrm{O}(|\overline{x}|^{1-n})  & \overline{x}\in \R^{n-1},\,\,n\geq 3\,, \\ 
\end{cases}
\end{equation}
for some $a,b \in \mathbb{R}$ and $\overline{c} \in \mathbb{R}^{n-1}$.
\end{corollary}

\begin{remark}[Applications of the asymptotic expansion]
    In the equal weights case, the proper chambers are bounded \cite[Theorem 2.4]{NovPaoTor23}. Therefore, the asymptotic expansion \eqref{eq:asymptotic exp} holds for locally minimizing $(N,2)$-clusters, and could be used for example in studying the analogue of the $N$-bubble conjecture for $(N-1,2)$-clusters.
\end{remark}

\begin{proof}
Since $\X$ is a $\bf c$-locally minimizing $(N,2)$-cluster in $\rn$ with $\cup_h \E(h) \subset B_R(0)$ for some $R>0$, $\X(N+1)$ and $\X(N+2)$ are perimeter minimizers in $\rn \setminus \cl B_{R}(0)$. Since they both have infinite volume, we can choose $R_m\to \infty$ such that $\pa B_{R_m}(0) \cap \X(N+1,N+2) \neq \varnothing$ for all $m$. Then Corollary \ref{corollary:tangent cones at infinity} implies that there exists a blow-down $\X_\infty$ which is a conical $\bf c$-locally minimizing $(0,M')$-cluster in $\rn$, for $M' \in \{1,2\}$. However, since $\Theta_\infty(\X)>0$, \eqref{eq:density of tangent cone} implies that we cannot have $M'=1$. So we have a locally minimizing, conical $(0,2)$-cluster in all of space: in other words, a pair of complementary entire perimeter minimizing cones. In $\rn$ for $n< 8$, since all entire perimeter minimizers are halfspaces, $\X_\infty(N+1)$ and $\X_\infty(N+2)$ are complementary halfspaces. By Corollary \ref{corollary:tangent cones at infinity}, this gives $\Theta_\infty(\X)=c_{(N+1)(N+2)}\omega_{n-1}$ when $n< 8$. In $\rn$ for $n \geq 8$, the same conclusion is a consequence of our assumption that $\Theta_\infty (\X) = c_{(N+1)(N+2)}\omega_{n-1}$ and Allard's theorem. We finish the proof in cases.

\medskip

\noindent\textit{The case $n=2$}:  Since $\X(N+1)$ and $\X(N+2)$ are perimeter minimizers in $\mathbb{R}^2 \setminus \cup_h \cl \E(h)$, $\X(N+1,N+2) \setminus \cl B_R(0)$ is a union of rays. Since $\Theta_\infty(\X)=c_{(N+1)(N+2)}\omega_{n-1}$, there can be only two such rays, and so every blow-down of $\X$ is the same pair of complementary halfspaces. Since $\X(N+1)$ and $\X(N+2)$ blow down to complementary halfspaces, each of them has one unbounded connected component, which we call $C_1$ and $C_2$, respectively. Let $\{C_\ell\}_{\ell > 2}$ be all the other connected components of the other chambers in addition to any bounded components of $\X(N+1)$, $\X(N+1)$, $\pa \X(N+1)$ and $\pa \X(N+2)$. Then $C_1$ and $C_2$ are perimeter minimizers in the open set $\mathbb{R}^2 \setminus \cl \cup_{\ell>2}C_\ell$, and $\pa C_1 \cap \pa C_2$ blows down to a line, which we take to be $\{x_2=0\}$ by rotating.
Thus we can find $r_i$, $a_i\in \mathbb{R}$ for $i=1,2$ such that $\cup_{\ell>2} C_\ell \subset [r_1,r_2]\times I$ for an interval $I\in \mathbb{R}$ and $\X(N+1,N+2)\setminus ([r_1,r_2] \times I)$ is the graph of a function $u:\mathbb{R}\setminus [r_1,r_2] \to \{a_1,a_2\}$ with
\begin{align}\notag
    u(\overline{x}) =
\begin{cases}
a_1 & \overline{x} < r_1 \\ 
a_2 & \overline{x} > r_2
\end{cases}
\end{align}
and each $(r_i,a_i) \in  \partial (\cup_{\ell > 2} C_\ell)$. We must show that $a_1=a_2$, as in \cite[Proof of Theorem 1.1, step five]{MagNov22} or \cite[Figure 3]{NovPaoTor23}.

\medskip

Let $s$ be the segment connecting $(r_1,a_1)$ and $(r_2,a_2)$. Let $r_1'=r_1-1$ and $r_2'=r_2+1$, and $s'$ be the segment joining $(r_1',a_1)$ and $(r_2',a_2)$. We create a new cluster $\X'$ in the following fashion.
\begin{itemize}
    \item Rotate $\cup_{\ell > 2} C_\ell$ around the midpoint of $s$ via a rotation map $R$ so that $R(r_i,a_i) \in s'$ for $i=1,2$.
    \item Let  $C' = R(\cl \cup_{\ell>2}C_\ell) \cup s'$, so that $C'\cup \{(t,a_1):t \leq r_1'\} \cup \{(t,a_2):t\geq r_2'\}$ disconnects $\mathbb{R}^2$ into two unbounded open sets. Call these $C_1'$ and $C_2'$ and add them to $\X'(N+1)$ and $\X'(N+2)$, respectively.
\end{itemize}
Direct computation shows that on any rectangle $\mathcal{R}=[r_1',r_2']\times I'$ for large interval $I'$,
\begin{align}\notag
    \pc(\X;&\mathcal{R}) -\pc(\X';\mathcal{R})\\ \notag
    &= c_{(N+1)(N+2)} \big[ 2 - \sqrt{(r_2-r_1+2)^2 + (a_1-a_2)^2} + \sqrt{(r_2-r_1)^2 + (a_1-a_2)^2} \big]\geq 0\,,
\end{align}
with equality if and only if $a_1=a_2$ by the triangle inequality. By the $\bf c$-local minimality of $\X$, we deduce that $a_1=a_2$.

\begin{figure}
\begin{overpic}[scale=0.65,unit=1mm]{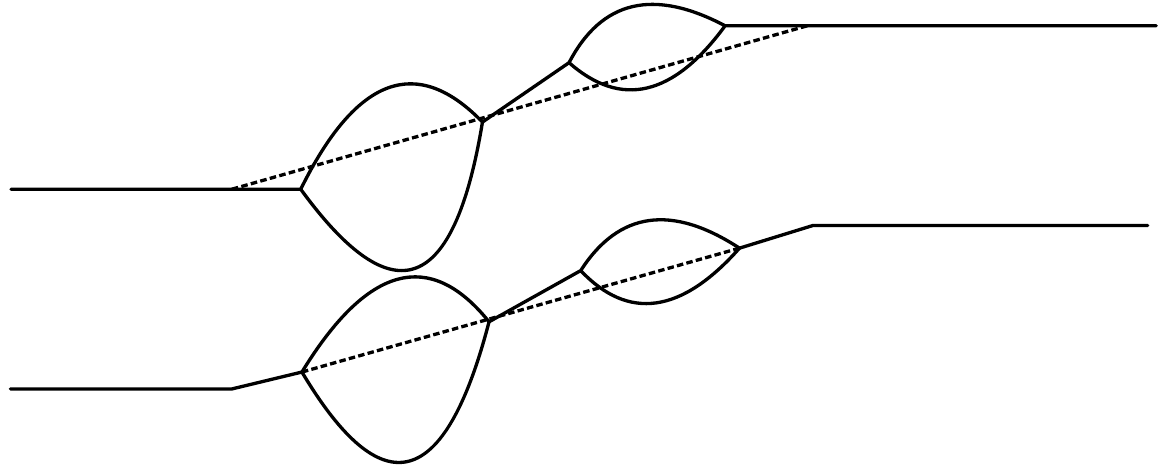}
\put(60.5,40){\small{$(r_2,a_2)$}}
\put(18,21.5){\small{$(r_1,a_1)$}}
\put(30,22){\small{$C_3$}}
\put(42,34){\small{$C_4$}}
\put(53,36){\small{$C_5$}}
\put(8,25.5){\small{$C_1$}}
\put(85,39.5){\small{$C_2$}}
\put(49,30){\small{$s'$}}
\end{overpic}
\caption{On top is the original configuration, and below is the rotated configuration with less energy. The dotted segment is $s'$, and $C_3$, $C_4$, and $C_5$ are the two components of $\X(1)$ and the segment connecting them. They are all rotated around the midpoint of $s'$ so that the rotated $(r_1,a_1)$ and $(r_2,a_2)$ lie on $s'$.}\label{fig:collinear}
\end{figure}
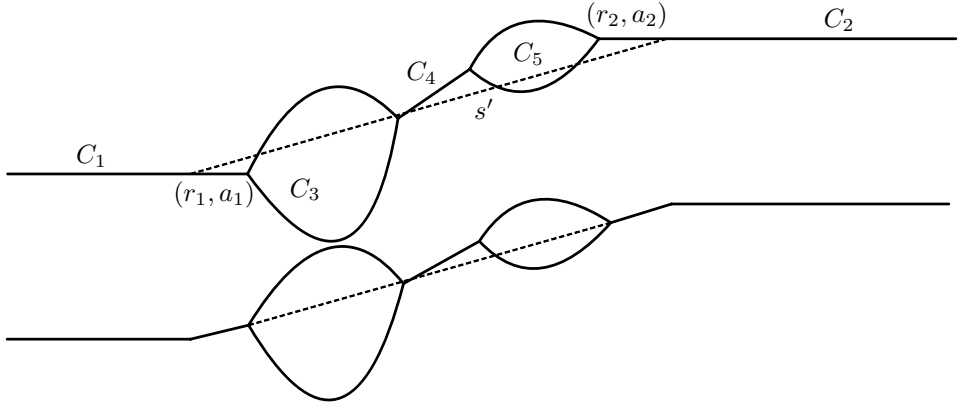

\medskip

\noindent\textit{The case $n\geq 4$}:
In this case, the rectifiable varifold $V=\var(\X(N+1,N+2),1)$ is stationary in $\rn \setminus B_R(0)$ for some $R>0$, has planar growth at infinity, and, for some Radon measure $\s$ on $\partial B_R(0)$ and unit-valued $\nu_V^{\rm co}:\partial B_R(0) \to \mathbb{S}^{n-1}$,
\begin{align}\label{eq:div thm 2}
    \delta V(X) = \int X(x) \cdot \nu^{\rm co}_V(x)\,d\sigma(x)\qquad \forall X \in C_c^\infty(\rn;\rn)
\end{align}
by \eqref{eq:div thm}. These are precisely the assumptions necessary to apply \cite[Theorem 2.1.$(ii)$]{MagNov22} and obtain uniqueness of the tangent plane at infinity to the varifold $V$. Up to a rotation, we can assume this tangent plane is $\{x_n=0\}$. Due to Allard's theorem and the convergence of $V/R$ to $\var(\{x_n=0\},1)$, $V$ can be parametrized as a minimal graph over $\{x_n=0\}$ outside some compact set.  By \cite[Proposition 4.1]{MagNov22}, which derives the asymptotic expansion for exterior solutions to the minimal surface equation following \cite[Propositions 1 and 3]{Sch83}, there exists $R_0>0$ such that $\X(N+1,N+2)\setminus C_{R_0}$ is the graph of $u:\mathbb{R}^{n-1}\setminus B_{R_0}^{n-1}(\overline{0})\to \mathbb{R}$ satisfying the minimal surface equation and
\begin{align}\label{eq:ngeq3 exp}
    u(\overline{x}) = a + \frac{b}{|\overline{x}|^{n-3}}+ \frac{\overline{c}\cdot \overline{x}}{|\overline{x}|^{n-1}}+ \mathrm{O}(|\overline{x}|^{1-n})\qquad |\overline{x}| \geq R_0\,.
\end{align}

\medskip

\noindent\textit{The case $n=3$}: The exact same argument as for $n\geq 4$ can be repeated verbatim, except that the expansion for the function $u$ has a logarithmic term which we must eliminate (cf. \cite[Proposition 4.1]{MagNov22}). So
\begin{align}\label{eq:n=3 exp}
    u(\overline{x}) = a + b \ln |\overline{x}|+ \frac{\overline{c}\cdot \overline{x}}{|\overline{x}|^{2}}+ \mathrm{O}(|\overline{x}|^{-2})\qquad |\overline{x}| \geq R_0
\end{align}
and we would like to show that $b=0$ using the minimality of $\X$. Assume for contradiction that $b\neq 0$. Then by the catenoidal growth in the expansion 
\eqref{eq:n=3 exp}, the energy of $\X$ inside large cylinders can be estimated from below by
\begin{align}\label{eq:area growth of catenoid}
\pc(\X; C_r) &\geq  c_{(N+1)(N+2)}\pi r^2 + C\ln r + \mathrm{O}(1)
\end{align}
for some $C>0$. To make our comparison cluster, for each $r$, let $\Pi_r$ be the plane $\{x_3 = a + b \ln|\bar x| \}$. Up to relabelling, we may assume that inside $C_r$, $\X(N+1)$ is ``below" $\X(N+2)$, in that $\X(N+1)$ is unbounded in the $-e_3$ direction and $\X(N+2)$ in the $e_3$ direction. We define a new cluster $\X_r$ by setting
\begin{align}\notag
    \X_r(j) &= \X(j) \quad \textup{if $1\leq j \leq N$}\\   \notag
     \X_r(N+1)&= (\{x:x_3<a+b \ln |\bar x|\}\cap C_r \setminus \cup_{j=1}^N \X_r(j)) \cup \X(N+1) \setminus C_r \\ \notag
      \X_r(N+2)&= (\{x:x_3>a+b \ln |\bar x|\}\cap C_r \setminus \cup_{j=1}^N \X_r(j)) \cup \X(N+2) \setminus C_r\,.
\end{align}
Then for large $r$, $\X(j) \Delta \X_r(j) \cc B_{r+1}^{2}(\overline{0})\times (-r,r)$. Also, by \eqref{eq:cut and paste} and \eqref{eq:n=3 exp} (which allow us to estimate the error coming from ``gluing" along $\partial C_r$), we can compute
\begin{align}\notag
    \pc(\X_r ; B_{r+1}^{2}(\overline{0}) \times (-r,r)) &\leq c_{(N+1)(N+2)} \pi r^2 + 2\pi r \sup_{|\bar x|=r}|\bar c \cdot \bar x|/r^2  \\ \notag
    &\qquad + \sup_{j<k} c_{jk} \sum_{j=1}^N P(\X(j))  \\ \label{eq:disk plus O1}
    &\leq c_{(N+1)(N+2)}\pi r^2 + \mathrm{O}(1)\,.
\end{align}
For large $r$, the combination of \eqref{eq:disk plus O1} and \eqref{eq:area growth of catenoid} contradicts the $\bf c$-local minimality of $\X$.
\end{proof}

\begin{lemma}[Equivalent notions of minimality]\label{lemma:equiv min notions}
    If $\bf c$ satisfies \eqref{eq:pos}-\eqref{eq:triangle} and an $(N,2)$-cluster $\X$ satisfies, for every $r>0$,
\begin{align}\label{eq:alt min def in lemma}
    \pc(\X;B_r(0)) \leq \pc(\X';B_r(0))
\end{align}
for any $\X'$ such that $\X(j) \Delta \X'(j) \cc B_r(0)$ and $|\X(j) \cap B_r(0)| = |\X'(j) \cap B_r(0)|$ for each $1\leq j \leq N+2$, then $\X$ is a $\bf c$-locally minimizing $(N,2)$-cluster in the sense of Definition \ref{def:loc mins}.
\end{lemma}

\begin{proof} The proof is divided into three main steps. After collecting a preliminary fact in step zero, in step one we prove an energy bound for $\X$ on any $B_r(x)$. Next, we use that bound and the fact that $|\X(N+1)| = \infty$ to locate balls $B_{R_k}(x_k)$ with $R_k\to \infty$ such that the volume fraction of $\X(N+1)$ on $B_{R_k}(x_k)$ is bounded away from $0$ and $1$. Finally, we use the volume fixing variations construction on $B_{R_k}(x_k)$ to fix the volumes of the improper chambers and prove \eqref{eq:alt min def in lemma}.

\medskip

\noindent{\it Step zero}: First, the local minimality condition \eqref{eq:alt min def in lemma} (with the volume of all the chambers being preserved on $B_r(0)$) is the same as the notion used in \cite[Theorem 3.1]{Leo01} (see \cite[pg. 8]{Leo01}), and so we may apply that theorem to find positive $\eta$ and $\rho$ such that
\begin{align}\label{eq:elimination us 1}
    \mbox{``if $r<\rho$ and $|B_r(x) \cap (\cup_{j\leq N} \X(j))|\leq \eta r^n$, then $(\cup_{j\leq N} \X(j))\cap B_{r/2}(x)=\varnothing$"}\,.
\end{align}
As a consequence of \eqref{eq:elimination us 1} and the fact that $|\X(j)| < \infty$ for $j\leq N$, $\X(j)$ is bounded if $j\leq N$. 

\medskip

\noindent{\it Step one}: Here we prove that if $x\in \rn$, $r>0$, and $\rho_i\nearrow r$, 
then for any $1\leq J\leq N+2$ such that $|\X(J) \cap B_r(x)|>0$,
\begin{align}\label{eq:basic energy bound}
    \pc(\X;B_r(x)) \leq C(N,n,{\bf c})\big[|B_r(x) \setminus \X(J)|^{(n-1)/n}+ \limsup_{i\to \infty}\hn(\pa B_{\rho_i}(x) \setminus \X(J)^\one) \big]\,.
\end{align}
For later use, we note that \eqref{eq:basic energy bound} also implies
\begin{align}\label{eq:basic energy bound 2}
 \pc(\X;B_r(x)) \leq C r^{n-1} \quad \mbox{for all $x\in \rn$, $r>0$}\,.   
\end{align}
Since $|\X(J)\cap B_r(x)|>0$ and $\rho_i \nearrow r$, by restricting to a tail of $\{\rho_i\}_i$ we may assume that $|\X(J) \cap B_{\rho_i}(x)|>0$ for all $i$. For each $i$, define a new cluster $\X_i$ by setting
\begin{align}\notag
    \X_i(j) \cap B_{\rho_i}(x) = A_j^i\,,\qquad \X_i(j) \setminus B_{\rho_i}(x) = \X(j) \setminus B_{\rho_i}(x)
\end{align}
where $A_j^i$ are pairwise disjoint open annuli centered at $x$ satisfying two conditions: first, $|A_j^i \cap B_{\rho_i}(x) | = |\X(j) \cap B_{\rho_i}(x)|$, and second that the outermost annulus is $A_J^i$. With $\X_i$ defined as such, the largest individual contribution to $\pc(\X;B_{\rho_i}(x))$ comes from the inner boundary of the outermost annulus $A_J^i$, which in turn depends on $|B_{\rho_i}(x) \setminus \X(J)|$. More precisely, we may estimate 
\begin{align}\notag
    \pc(\X_i;B_{\rho_i}(x)) &\leq \sup_{j<k}c_{jk}\sum_{j} \hn(\pa A_j^i \cap B_{\rho_i}(x)) \\ \label{eq:rhok inside energy}
    &\leq (N+1) (\sup_{j<k} c_{jk} )n\,\omega_n^{1/n}|B_{\rho_i}(x) \setminus \X(J)|^{(n-1)/n}\,.
\end{align}
Now since $A_J^i\neq \varnothing$ and borders $\pa B_{\rho_i}(x)$, the Lebesgue density of $A_J^i$ and thus $\X_i(J)$ is at least $1/2$ for all $y\in \pa B_{\rho_i}(x)$. Therefore by \eqref{eq:hn partition} and $\pa^* \X_i(J) \ehn \X_i(J)^\half$, $\pa B_{\rho_i}(x) \overset{\mathcal{H}^{n-1}}{\subset} \X_i(J)^\one \cup \pa^* \X_i(J)$. Furthermore, the Lesbesgue density of $\X_i(J)$ at $y\in \pa B_{\rho_i}(x)$ is at least that of $\X(J)$ at $y$, since $\X_i(J) = \X(J) \cup A_J^i$ in a neighborhood of $\pa B_{\rho_i}(x)$. Thus $\X(J)^\one \cap \pa B_{\rho_i}(x) \subset \X_i(J)^\one \cap \pa B_{\rho_i}(x)$. Putting these observations together, we estimate
\begin{align}\label{eq:pa rhok calc}
    \pc(\X_i;\pa B_{\rho_i}(x)) \leq  \sup_{j<k}c_{jk}\hn(\pa B_{\rho_i}(x) \setminus \X_i(J)^\one)\leq \sup_{j<k}c_{jk}\hn(\pa B_{\rho_i}(x) \setminus \X(J)^\one)\,.
\end{align}
Then by the minimality \eqref{eq:alt min def in lemma} and \eqref{eq:rhok inside energy}-\eqref{eq:pa rhok calc},
\begin{align}\notag
\pc(\X;B_r(x)) \leq \pc(\X_i;B_r(x)) &= \pc(\X_i;B_{\rho_i}(x)) + \pc(\X_i;\pa B_{\rho_i}(x)) + \pc(\X;B_r(x) \setminus \cl B_{\rho_i}(x)) \\ \notag 
&\leq C(N,n,{\bf c})\big[|B_{\rho_i}(x) \setminus \X(J)|^{(n-1)/n}+ \hn(\pa B_{\rho_i}(x) \setminus \X(J)^\one) \big] \\ \notag
&\qquad + \pc(\X;B_r(x) \setminus \cl B_{\rho_i}(x))\,.
\end{align}
Since $\X(j)$ are sets of locally finite perimeter, $\pc(\X;B_r(x) \setminus \cl B_{\rho_i}(x))\to 0$ as $\rho_i \to r$. Therefore, sending $i\to \infty$ concludes the proof of \eqref{eq:basic energy bound}.

\medskip

\noindent{\it Step two}: In this step we show the following claim, which we use to finish the proof in step three. 

\medskip

\noindent{\it Claim}: There exists $\delta\in (0,1/2)$ and sequences $R_k\to \infty$ and $\{x_k\}_k \subset \rn$ such that
\begin{align}\label{eq:big ball density bounds}
   \delta \omega_n R_k^n \leq |\X(N+1) \cap B_{R_k}(x_k)| \leq (1-\delta) \omega_n R_k^n\quad \forall k\,.
\end{align}

\medskip

\noindent{\it To prove \eqref{eq:big ball density bounds}}: If \eqref{eq:big ball density bounds} holds with $x=0$ and any sequence $R_k$ diverging to infinity, we are done. So it suffices to consider the case that
\begin{align}\label{eq:volume ratio dies}
   \min\big\{|\X(N+1) \cap B_{R}(0)|/(\omega_nR^n)\,\,,\,\,| B_R(0) \setminus \X(N+1)|/(\omega_nR^n) \big\} \to 0\quad \mbox{ as }R\to \infty\,. 
\end{align}
Since $R \mapsto |\X(N+1) \cap B_R(0)|/(\omega_nR^n)$ is continuous in $R$, if \eqref{eq:volume ratio dies} holds, then actually only one of $|\X(N+1) \cap B_{R}(0)|/(\omega_nR^n)\to 0$ or $| B_R(0) \setminus \X(N+1)|/(\omega_nR^n)\to 0$ holds along the whole sequence. Furthermore, since $|\X(j)|<\infty$ for $j\notin \{N+1,N+2\}$, the former case entails $| B_R(0) \setminus \X(N+2)|/(\omega_nR^n)\to 0$, while the latter entails $|\X(N+2) \cap B_{R}(0)|/(\omega_nR^n)\to 0$. So by interchanging the labels on $\X(N+1)$ and $\X(N+2)$ if necessary, we may as well assume that
\begin{align}\label{eq:n+1 volume ratio dies}
    |\X(N+1) \cap B_R(0)|/ (\omega_n R^n) \to 0\quad \mbox{as }R\to \infty
\end{align}
and prove \eqref{eq:big ball density bounds}. The idea is to use the energy bound \eqref{eq:basic energy bound} to find $R_k\to \infty$ such that
\begin{align}\label{eq:scale invariant perimeter bound}
P(\X(N+1) \cap B_{R_k}(0)) \lesssim |\X(N+1) \cap B_{R_k}(0)|^{(n-1)/n}\,.
\end{align}
\eqref{eq:scale invariant perimeter bound} implies that the rescaled unit volume sets $E_k:=(\X(N+1) \cap B_{R_k}(0))/|\X(N+1) \cap B_{R_k}(0)|^{1/n}$ have uniformly bounded perimeters. Such a uniform perimeter bound is only possible if these rescalings are not too ``sparse"; that is, $\max_{x\in \rn} |E_k \cap B_1(x)|$ is bounded away from $0$ uniformly in $k$. These ``good" balls on which the volume fraction of $E_k$ is non-negligible will rescale to be our $B_{R_k}(x_k)$, with $R_k\to \infty$ since $|\X(N+1)|=\infty$.

\medskip

First, we obtain a rewritten version of \eqref{eq:basic energy bound}. Let $m(r) = |B_r(0) \cap \X(N+1)|$. By the co-area formula, $m$ is differentiable almost everywhere, with derivative belonging to $L^1_\loc(\mathbb{R})$ and given by $m'(r) = \hn(\X(N+1)^\one \cap \pa B_r(0))$ for almost every $r$. Then almost every $r$ is a Lebesgue point of $m'$. For such a Lebesgue point $r$, we can find a sequence $\rho_i\nearrow r$ as $i\to \infty$ such that $m'(\rho_i) \to m'(r)$ as $i\to \infty$. Inserting these $\rho_i$ into the estimate \eqref{eq:basic energy bound} with $J=N+2$, we obtain for almost every $r>0$
\begin{align}\label{eq:rk bound}
   \pc(\X;B_{r}(0)) \leq C(N,n,{\bf c})\big[|B_{r}(0) \setminus \X(N+2)|^{(n-1)/n}+ \hn(\pa B_{r}(0) \setminus \X(N+2)^\one) \big]\,. 
\end{align}
Now since $|\X(N+1)|=\infty$ and $|\X(j)|<\infty$ for $1\leq j \leq N$, we can choose $R_0>0$ such that
\begin{align}\label{eq:comparable N+1 and everybody else}
    |B_r(0) \setminus \X(N+2)|^{(n-1)/n} \leq 2|B_r(0) \cap \X(N+1)|^{(n-1)/n}\quad \forall r>R_0\,.
\end{align}
Also, by the boundedness of the proper chambers from step zero and \eqref{eq:hn partition}, there exists $R_1>R_0$ such that
\begin{align}\label{eq:these slices dont see propers}
    \pa B_r(0) \shn \X(N+1)^\one \cup \X(N+2)^\one \cup \X(N+1,N+2)\quad \forall r>R_1\,.
\end{align}
Since $\hn \mres \X(N+1,N+2)$ is locally finite, $\hn(\X(N+1,N+2) \cap \pa B_r(0)) = 0$ for all but countably many $r>0$. So
\begin{align}\label{eq:in the ones zero}
    \pa B_r(0) \shn \X(N+1)^\one \cup \X(N+2)^\one\quad \mbox{for a.e.~ $r>R_1$}\,,
\end{align}
in which case
\begin{align}\label{eq:in the ones}
    \hn(\pa B_r(0) \setminus \X(N+2)^\one ) = \hn (\pa B_r(0) \cap \X(N+1)^\one)\quad \mbox{for a.e.~ $r>R_1$}\,.
\end{align}
By plugging \eqref{eq:comparable N+1 and everybody else} and \eqref{eq:in the ones} into \eqref{eq:rk bound}, for almost every $r>R_1$ we have
\begin{align}\label{eq:rk bound v2}
    \pc(\X;B_{r}(0)) \leq C(N,n,{\bf c}) \big[2|B_{r}(0) \cap \X(N+1)|^{(n-1)/n} + \hn(\pa B_{r}(0) \cap \X(N+1)^\one)\big]
\end{align}
as well as \eqref{eq:in the ones zero}. As a last modification of this energy bound, by \eqref{eq:pos} and \eqref{eq:hn partition} we have 
$$
P(\X(N+1);B_{r}(0)) \leq \max c_{jk}^{-1} \pc(\X;B_{r}(0))\,.
$$
Thus up to changing the values of the multiplicative constants in \eqref{eq:rk bound v2}, for almost every $r>R_1$,
\begin{align}\label{eq:rk bound v3}
    P(\X(N+1);B_{r}(0)) &\leq C\big[|B_{r}(0) \cap \X(N+1)|^{(n-1)/n} + \hn(\pa B_{r}(0) \cap \X(N+1)^\one)\big]\mbox{ and } \\ \label{eq:rk bound v3a}
    \pa B_r(0) &\shn \X(N+1)^\one \cup \X(N+2)^\one\,.
\end{align}

\medskip

We claim now that we can choose $r_k\to \infty$ such that \eqref{eq:rk bound v3}-\eqref{eq:rk bound v3a} hold for all $k$ and
\begin{align}\label{eq:boundary dominates inside}
    (C+1/2) \hn(\pa B_{r_k}(0) \cap \X(N+1)^\one) \leq P(\X(N+1);B_{r_k}(0))/2\,.
\end{align}
Indeed, supposing momentarily for the sake of contradiction that this is not possible, then there would be $R_2>R_1$ such that for almost every $r>R_2$, $P(\X(N+1);B_r(0))/2< (C+1/2) \hn(\pa B_r(0) \cap \X(N+1)^\one) $. Adding $\hn(\pa B_r(0)\cap \X(N+1)^\one)/2$ to both sides of this inequality and using \eqref{eq:cut and past intersection} (which applies by \eqref{eq:rk bound v3a}) gives
\begin{align}\notag
    P(\X(N+1) \cap B_r(0))/2 &= P(\X(N+1);B_r(0))/2 + \hn(\X(N+1)^\one \cap \pa B_r(0))/2 \\ \label{eq:bad diff ineq}
    &\leq (C+1) \hn(\X(N+1)^\one \cap \pa B_r(0)) \quad \mbox{for a.e. $r>R_2$}\,.
\end{align}
By recalling that $m'(r) = \hn(\X(N+1)^\one \cap \pa B_r(0))$ for almost every $r$ and using the isoperimetric inequality on the left hand side of \eqref{eq:bad diff ineq}, this inequality rewrites as
\begin{align}\label{eq:bad diff ineq 2}
    n\omega_n^{1/n}m(r)^{(n-1)/n} /2\leq (C+1) m'(r) \quad \mbox{for a.e.~ $r>R_2$}\,.
\end{align}
Since $|\X(N+1)|=\infty$, then $m(r)>0$ on some interval $(R_3,\infty)\subset (R_2,\infty)$, and we may integrate \eqref{eq:bad diff ineq 2} on $(R_3,\rho)$ to find
\begin{align}\label{eq:bad diff ineq 3}
    C_1\rho + C_2 \leq m(\rho)^{1/n} \quad \mbox{for a.e.~ $\rho>R_3$}
\end{align}
and some constants $C_1$ and $C_2$ with $C_1>0$. But \eqref{eq:bad diff ineq 3} implies that $m(\rho)$ grows like $\rho^n$ as $\rho\to \infty$, which contradicts \eqref{eq:n+1 volume ratio dies}. Concluding our contradiction argument for \eqref{eq:rk bound v3}, \eqref{eq:rk bound v3a}, and \eqref{eq:boundary dominates inside}, it must be the case that we can in fact choose $r_k\to \infty$ such that \eqref{eq:rk bound v3}, \eqref{eq:rk bound v3a}, and \eqref{eq:boundary dominates inside} hold for all $r_k$. 

\medskip

Now for these $r_k\to \infty$ we have chosen, we add $\hn(\X(N+1)^\one \cap \pa B_{r_k}(0))/2$ to both sides of \eqref{eq:rk bound v3}, yielding
\begin{align}\notag
    P(\X(N+1);B_{r_k}(0)) &+ \hn(\X(N+1)^\one \cap \pa B_{r_k}(0))/2  \\ \notag
    &\leq C|B_{r_k}(0) \cap \X(N+1)|^{(n-1)/n} + (C+1/2)\hn(\pa B_{r_k}(0) \cap \X(N+1)^\one)\,.
\end{align}
Due to the inequality \eqref{eq:boundary dominates inside}, this simplifies to
\begin{align}\notag
    \big[ P(\X(N+1);B_{r_k}(0)) &+ \hn(\X(N+1)^\one \cap \pa B_{r_k}(0))\big]/2 \leq C|B_{r_k}(0) \cap \X(N+1)|^{(n-1)/n}\,.
\end{align}
By the containment in \eqref{eq:rk bound v3a} and \eqref{eq:cut and past intersection}, the left hand side is $P(\X(N+1) \cap B_{r_k}(0))/2$, so we have
\begin{align}\label{eq:rk bound v4}
    P(\X(N+1) \cap B_{r_k}(0))\leq 2C|B_{r_k}(0) \cap \X(N+1)|^{(n-1)/n}\,.
\end{align}
Now, setting $s_k = |\X(N+1) \cap B_{r_k}(0))|^{1/n}$ we define the rescaled sets of finite perimeter 
$$
E_k = (\X(N+1) \cap B_{r_k}(0))/s_k\,.
$$
By the definition of $s_k$, the isoperimetric inequality and \eqref{eq:rk bound v4}, and $|\X(N+1)| = \infty$ and \eqref{eq:n+1 volume ratio dies},
\begin{align}\label{eq:Ek unit volume}
    &|E_k| = 1\,,\quad E_k \subset B_{r_k/s_k}(0)\,, \\ \label{eq:Ek unit perimeter}
    &C(n) \leq P(E_k) = s_k^{1-n}P(\X(N+1) \cap B_{r_k}(0)) \leq 2C \quad \mbox{for some $C(n)$}\,,\\ \label{eq:Ek in big balls}
    &s_k  \to \infty\quad\mbox{and}\quad s_k/r_k\to 0\,.
\end{align}
By the statement of Almgren's nucleation lemma \cite[VI.13]{Alm76} from \cite[Lemma 29.10]{Mag12}, \eqref{eq:Ek unit volume}-\eqref{eq:Ek unit perimeter} yield the existence of a dimensional constant $c(n)$ such that, setting $\varepsilon = \min \{1 ,  C(n) / (2nc(n)) \}/2$, there exist $y_k \in \rn$ such that
\begin{align}\label{eq:vol dens from nucleation}
    &|E_k \cap B_1(y_k)|  \geq \bigg(c(n) \frac{\varepsilon}{P(E_k)} \bigg)^n \geq  \bigg(c(n) \frac{\varepsilon}{2C} \bigg)^n\,.
\end{align}
Set 
$$\delta = \min\{(c(n)\e/(2C))^n/\omega_n,1/4 \}\,.
$$
Rescaling back by $s_k$ and setting $x_k =s_k y_k$, \eqref{eq:vol dens from nucleation} says that in terms of $\X(N+1) \cap B_{r_k}$,
\begin{align}\label{eq:vol dens from nucleation 2}
    |\X(N+1) \cap B_{s_k}(x_k)| \geq |\X(N+1) \cap B_{r_k}(0) \cap B_{s_k}(x_k)| \geq \delta \omega_n s_k^{n}\,,
\end{align}
where $\delta \in (0,1/2)$. 

\medskip

We may now finish the proof of \eqref{eq:big ball density bounds}. If
\begin{align}\label{eq:good upper density}
|\X(N+1) \cap B_{s_k}(x_k)|\leq (1-\delta)\omega_n s_k^n\,,
\end{align}
then we set $s_k=R_k$, and \eqref{eq:big ball density bounds} follows from \eqref{eq:vol dens from nucleation 2}, \eqref{eq:good upper density}, and \eqref{eq:Ek in big balls}. Suppose instead that, up to a subsequence,
\begin{align}\label{eq:bad upper density}
|\X(N+1) \cap B_{s_k}(x_k)|> (1-\delta)\omega_n s_k^n\,.
\end{align}
We observe that since $\X(N+1) \cap B_{r_k}(0) \cap B_{s_k}(x_k)\neq \varnothing$ and $s_k/r_k\to 0$ (from \eqref{eq:Ek in big balls}), it must be the case that $x_k \in B_{2r_k}(0)$ for large $k$. By combining this containment with \eqref{eq:n+1 volume ratio dies}, we estimate
\begin{align}\label{eq:n+1 volume ratio dies 2}
    \frac{|\X(N+1) \cap B_{r_k}(x_k)|}{\omega_n r_k^n} \leq 3^n \frac{|\X(N+1) \cap B_{3r_k}(0)|}{\omega_n (3r_k)^n} \to 0\quad\mbox{as }k\to \infty\,.
\end{align}
By the continuity of $r\mapsto |\X(N+1) \cap B_r(x_k)|/(\omega_nr^n)$, \eqref{eq:bad upper density} and \eqref{eq:n+1 volume ratio dies 2} allow us to choose $R_k \in (s_k,r_k)$ for large $k$ such that
\begin{align}\label{eq:good density final}
    |\X(N+1) \cap B_{R_k}(x_k)| = \omega_n R_k^n/2\,,
\end{align}
which also implies \eqref{eq:big ball density bounds}.

\medskip

\noindent{\it Step three}: By rescaling the density bound \eqref{eq:big ball density bounds} and the energy bound \eqref{eq:basic energy bound 2}, the sets $F_k:=(\X(N+1)-x_k)/R_k$ satisfy
\begin{align}\label{eq:density and energy bound}
   \delta \omega_n \leq |F_k \cap B_1(0)| \leq (1-\delta)\omega_n \quad\mbox{and}\quad \sup_k P(F_k;B_r(0)) \leq Cr^{n-1}\quad \forall r>0\,.
\end{align}
After restricting to a further subsequence and using the facts that $\cup_{j\leq N}\X(j)\cc B_R(0)$ for some $R>0$ and $R_k\to \infty$, we may choose $z\in \cl B_1(0)$ and $t_k\to 0$ such that
\begin{align}\label{eq:propers near z}
    B_1(0) \cap  \big[(\cup_{j\leq N}\X(j) - x_k)/R_k\big]\subset B_1(0) \cap \big[(B_R(0) -x_k)/R_k\big] \subset  B_{t_k}(z)\quad \mbox{for large $k$}\,;
\end{align}
the left hand side is empty if $\cup_{j\leq N}\X(j) \cap B_{R_k}(x_k) = \varnothing$. From the compactness for sets of finite perimeter, there exists a limiting set of locally finite perimeter $F$ that also satisfies the density bound on $B_1(0)$ from \eqref{eq:density and energy bound}. In particular, $\pa^* F \cap B_1(0)\neq \varnothing$, which means that we can choose $x\in \pa^* F \cap B_1(0)$ and $t>0$ such that $B_{t_k}(z)\cap B_t(x)=\varnothing$ for large $k$. Thus for large $k$,
\begin{align}\notag
    \cup_{j\leq N}\X(j) \cc B_R(0) \subset B_{R_kt_k}(R_kz+x_k)  \subset B_{R_kt}(R_kx+x_k)^c\,.
\end{align}
By the volume fixing variations construction \cite[Theorem 29.14]{Mag12} applied to $\{F_k\}_k$, we may obtain $\sigma_0>0$, $C>0$, and $K_0\in \mathbb{N}$ such that for any $\sigma\in [-\sigma_0,\sigma_0]$ and $k\geq K_0$, there exists $F_k^\sigma$ such that $F_k^\sigma \Delta F_k \cc B_{t}(x)$, 
\begin{align}\notag
  &|F_k^\sigma \cap B_{t}(x)| = |F_k \cap B_{t}(x)| + \sigma \qquad \mbox{and}\\ \notag
  &P(F_k^\sigma;B_{t}(x))\leq P(F_k;B_{t}(x)) + C |\sigma|\,. 
\end{align}

Now let $\X'$ be an $(N,2)$-cluster such that $\X'(j)\Delta \X(j) \subset B_r(0)$ for some $r>0$ and $|\X(j)| = |\X'(j)|$ for each $j$. We will be done if we prove \eqref{eq:alt min def in lemma} on $B_{r}(0)$; note that we may as well assume that $r>R$ (recalling that $B_R(0)$ contains the proper chambers of $\mathcal{X}$), since  the energy inequality on a smaller ball can be deduced from the analogous one on a larger ball by cancelling like terms. Let $\sigma_k = (|\X(N+1)\cap B_r(0)| - |\X'(N+1) \cap B_r(0)| )/ R_k^n$, and define the $(N,2)$-clusters $\X_k$ for $k$ large enough such that $B_r(0) \cap B_{R_kt}(R_kx + x_k)=\varnothing$ (in other words, the rescaled and translated volume fixing variations $R_kF_k^\sigma + x_k$ do not disturb the proper chambers) and $|\sigma_k| \leq \sigma_0$ by
\begin{align}\notag
    \X_k(j) &= \X'(j)\,\, \mbox{ for }j\leq N\,,\\ \notag
    \X_k(N+1) &= \big[\X'(N+1) \cap B_{r}(0)\big] \cup \big[(R_kF_k^{\sigma_k} + x_k) \cap B_{R_kt}(R_kx+x_k)\big]\\ \notag
    &\qquad\cup \big[\X(N+1) \cap B_r(0)^c \cap B_{R_kt}(R_kx+x_k)^c\big] \\ \notag
    \X_k(N+2) &= \big[\X'(N+2) \cap B_{r}(0)\big] \cup \big[B_{R_kt}(R_kx+x_k) \setminus (R_kF_k^{\sigma_k} + x_k)\big]\\ \notag
    &\qquad \cup \big[\X(N+2) \cap B_r(0)^c \cap B_{R_kt}(R_kx+x_k)^c\big]\,.
\end{align}
Then $\X_k(j) \Delta \X(j) \cc B_{r}(0) \cup B_{R_kt}(x_k + R_kx)$ for each $j$. Let $B_{s_k}(0)$ be large balls containing $B_r(0) \cup B_{R_kt}(x_k + R_kx)$. Due to the facts that for $j\leq N$, $\X(j) \cc B_R(0) \subset B_r(0)$, $|\X(j)| = |\X'(j)|$, and $\X'(j)\Delta \X(j) \cc B_r(0)$, we can compute
\begin{align}\notag
    |\X_k(j) \cap B_{s_k}(0)|= |\X_k(j) \cap B_r(0)| = |\X(j) \cap B_r(0)| = |\X(j) \cap B_{s_k}(0)|\quad \forall j\leq N\,.
\end{align}
Furthermore, using in order $\X_k(N+1) \Delta \X(N+1) \cc B_{r}(0) \cup B_{R_kt}(x_k + R_kx)$ and the definition of $F_k^{\sigma_k}$ followed by the definition of $F_k$ and our choice of $\sigma_k$,
\begin{align}\notag
  |\X_k(N+&1) \cap B_{s_k}(0)|\\ \notag
  &= |\X'(N+1) \cap B_{r}(0)|+ |(R_k F_k^{\sigma_k}+x_k) \cap B_{R_kt}(x_k + R_kx) | \\ \notag 
  &\quad +|\X(N+1)\cap B_{s_k}(0) \cap B_r(0)^c \cap B_{R_kt}(x_k+R_kx)^c| \\ \notag
  &= |\X'(N+1) \cap B_{r}(0)| + R_k^n\big(| F_k\cap B_{t}(x)| + \sigma_k\big)\\ \notag 
  &\quad +|\X(N+1)\cap B_{s_k}(0) \cap B_r(0)^c \cap B_{R_kt}(x_k+R_kx)^c|  \\ \notag
  &= |\X'(N+1) \cap B_{r}(0)| + |\X(N+1) \cap B_{R_kt}(x_k+R_kx)| + |\X(N+1)\cap B_r(0)| \\ \notag
  &\quad - |\X'(N+1) \cap B_r(0)|+|\X(N+1)\cap B_{s_k}(0) \cap B_r(0)^c \cap B_{R_kt}(x_k+R_kx)^c|\\ \notag
  &=|\X(N+1) \cap B_{s_k}(0)|\,.
\end{align}
So $|\X_k(j) \cap B_{s_k}(0)| = |\X(j) \cap B_{s_k}(0)|$ for $j \leq N+1$, and thus for $j=N+2$ also since $\{\X(j) \cap B_{s_k}(0)\}_j$ partition $B_{s_k}(0)$. We may therefore test the volume-constrained minimality of $\X$ against $\X_k$ on $B_{s_k}(0)$ and cancel like terms, yielding
\begin{align}\notag
    \pc(\X;B_r(0)) &+ c_{(N+1)(N+2)}P(\X(N+1);B_{R_kt}(x_k+R_kx))  \\ \notag
    &\leq \pc(\X';B_r(0)) + c_{(N+1)(N+2)}P(R_kF_k^{\sigma_k}+x_k;B_{R_kt}(x_k+R_kx)) \\ \notag
    &\leq \pc(\X';B_r(0)) +c_{(N+1)(N+2)}P(\X(N+1);B_{R_kt}(x_k+R_kx)) \\ \notag
    &\qquad +  R_k^{n-1}C \big||\X(N+1)\cap B_r(0)| - |\X'(N+1) \cap B_r(0)|\big| R_k^{-n}\,.
\end{align}
By subtracting $c_{(N+1)(N+2)}P(\X(N+1);B_{R_kt}(x_k+R_kx))$ from both sides and sending $k\to \infty$, we conclude that $\pc(\X;B_r(0)) \leq \pc(\X';B_r(0))$. Thus $\X$ is $\bf c$-locally minimizing in the sense of Definition \ref{def:loc mins}.
\end{proof}

\noindent In \cite{NovPaoTor23}, the authors prove a closure theorem \cite[Theorem 2.17]{NovPaoTor23} for $\bf c$-locally minimizing clusters when $\bf c$ is a family of equal positive weights under the assumption that for every pair of indices $j\neq k$ corresponding to improper chambers of the limiting cluster $\X$, there exists an $(n-1)$-dimensional halfspace contained in $\X(j,k)$ \cite[Definition 2.15]{NovPaoTor23}. For the case where there are two improper chambers, the following corollary of Lemma \ref{lemma:equiv min notions} removes this assumption on the limiting cluster.

\begin{corollary}[Closure for locally minimizing clusters]\label{corollary:closure theorem}
    If ${\bf c}_{\rm Id}$ is the family of equal unit weights, $\{\X_m\}_m$ is a sequence of $(N',N+2-N')$-clusters for some $N'<N+2$ such that, for every open $A \cc \rn$, 
\begin{align}\label{eq:alt min corollary}
    P_{{\bf c}_{\rm Id}}(\X_m;A)\leq P_{{\bf c}_{\rm Id}}(\X';A) 
\end{align}
    whenever $\X'(j) \Delta \X_m(j) \cc A$ and $|\X'(j) \cap A| = |\X_m(j)\cap A|$ for $1\leq j \leq N+2$, and there exists an $(N,2)$-cluster $\X$ such that $\X_m(j) \overset{\loc}{\to}\X(j)$ for $1\leq j \leq N+2$, then $\X$ is a ${\bf c}_{\rm Id}$-locally minimizing cluster in the sense of Definition \ref{def:loc mins}.
\end{corollary}

\begin{proof}
In the terminology of Novaga-Paolini-Tortorelli \cite[Definition 2.12]{NovPaoTor23}, the minimality condition \eqref{eq:alt min corollary} means that each $\X_m$ is a locally $J$-isoperimetric $(N+2)$-partition of $\rn$, where $J=\{1,\dots, N+2\}$ refers to the indices of the chambers subject to volume constraints. By the closure theorem for $J$-isoperimetric partitions \cite[Theorem 2.13]{NovPaoTor23}, $\X$ is a locally $J$-isoperimetric $(N+2)$-partition of $\rn$, or in the language of Lemma \ref{lemma:equiv min notions}, an $(N,2)$-cluster satisfying \eqref{eq:alt min def in lemma} for the energy $P_{{\bf c}_{\rm Id}}$. By the conclusion of Lemma \ref{lemma:equiv min notions}, $\X$ is a ${\bf c}_{\rm Id}$-locally minimizing $(N,2)$-cluster.
\end{proof}

\noindent The previous corollary should hold for any family of weights $\bf c$ satisfying \eqref{eq:pos}-\eqref{eq:triangle} as well, although we do not pursue this here.

\section{Proof of Theorems \ref{thm:local minimality of lens} and \ref{thm:double bubble}}\label{sec:main proof}

\begin{proof}[Proof of Theorem \ref{thm:local minimality of lens}]
    Let $\X$ be a $(1,2)$-cluster such that $\X(j) \Delta \xl(j) \cc B_r(0)$ for some $r>0$ and each $j$, and $|\X(1)|=1$. We must show that $\pc(\xl;B_r(0))\leq \pc(\X;B_r(0))$. Since $\xl(1)$ is bounded by its definition, we choose $R>r$ such that $\xl(1) \cc B_R(0)$ and claim it is enough to show that 
\begin{align}\label{eq:desired minimality inequality}
    \pc(\xl;B_R(0)) \leq \pc(\X;B_R(0))\,.
\end{align}
   Indeed, by our assumption on $\X$, $\pc(\xl ;B_R(0) \setminus B_r(0)) = \pc(\X;B_R(0) \setminus B_r(0))$, so \eqref{eq:desired minimality inequality} gives the desired minimality on $B_r(0)$. For each $m>0$, let $B_{r_m}(y_m)$ be the ball such that $\xbub(2,3)\subset \pa B_{r_m}(y_m)$, where $\xbub$ is the standard weighted double bubble from Definition \ref{def:standard weighted double bubble} with volume vector $(1,m,\infty)$. We set
\begin{align}\label{eq:vol change 0}
    \sigma_m = |\xbub(2) \cap B_R(0)| - |\X(2) \cap B_R(0)|\,.
\end{align}
By the volume fixing variations construction \cite[Lemma 17.21]{Mag12}, there exists $\Lambda>0$ and a family $A_m$ of diffeomorphic images of $B_1(0)$ with $A_m \Delta B_1(0) \cc B_{1/2}(e_n)$ for large $m$ satisfying
\begin{align}\label{eq:vol change}
    |A_m| &= |B_1(0)| + \sigma_m / r_m^n \quad\mbox{and} \\ \label{eq:MC}
    P(A_m;B_{1/2}(e_n)) &\leq P(B_1(0);B_{1/2}(e_n)) + \Lambda |\sigma_m| / r_m^{n} \,.
\end{align}
By the fact that $B_1(0) \setminus B_{1/2}(e_n) = A_m \setminus B_{1/2}(e_n)$, the sets $r_m A_m + y_m$ satisfy
\begin{align}\label{eq:agrees with xbub} 
    (r_m A_m + y_m) \setminus B_{r_m/2}(y_m + r_me_n) &= B_{r_m}(y_m) \setminus  B_{r_m/2}(y_m + r_me_n)\,.  
\end{align}
We also have, by the definition of $B_{r_m}(y_m)$,
\begin{align}\label{eq:agree outside}
    B_{r_m}(y_m) \setminus B_R(0) = \xbub(2) \setminus B_R(0)\,.
\end{align}
Next, we note that $y_m \cdot e_n \to \infty$ as $m\to \infty$ by the definition of $\xbub$, and so
\begin{align}\label{eq:no ball overlap}
\cl B_R(0) \cap B_{r_m/2}(y_m+r_me_n)= \varnothing \quad \mbox{for large $m$}\,.
\end{align}
As a consequence of \eqref{eq:agrees with xbub} and \eqref{eq:no ball overlap}, $B_{r_m}(y_m)\cap B_R(0)=(r_mA_m+y_m) \cap B_R(0)$ for large $m$. By this equivalence, \eqref{eq:vol change} rescaled by $r_m^n$, and \eqref{eq:agree outside}, we have
\begin{align}\notag
    |(r_m A_m+y_m) \setminus B_R(0)| &=|(r_mA_m+y_m)|- |B_{r_m}(y_m) \cap B_R(0)| \\ \label{eq:vol change 2} 
    &=|B_{r_m}(y_m)| + \sigma_m - |B_{r_m}(y_m) \cap B_R(0)|=|\xbub(2) \setminus B_R(0)| + \sigma_m\,.
\end{align}
Next, using in order \eqref{eq:agree outside}, \eqref{eq:no ball overlap} to split $\cl B_R(0)^c$, \eqref{eq:agrees with xbub} to cancel equal terms, and \eqref{eq:MC} rescaled by $r_m^{n-1}$, we estimate
\begin{align}\notag
    P(r_mA_m+y_m&; \cl B_R(0)^c) - P(\xbub(2); \cl B_R(0)^c) \\ \notag
    &=P(r_mA_m+y_m; \cl B_R(0)^c) - P(B_{r_m}(y_m); \cl B_R(0)^c) \\ \notag
    &= P(r_mA_m+y_m; B_{r_m/2}(y_m+r_me_n)) - P(B_{r_m}(y_m); B_{r_m/2}(y_m+r_me_n)) \\ \notag
    &\qquad + P\big(r_mA_m+y_m;B_{r_m/2}(y_m+r_me_n)^c\cap \cl B_R(0)^c \big) \\ \notag
    &\qquad -P\big(B_{r_m}(y_m);B_{r_m/2}(y_m+r_me_n)^c\cap \cl B_R(0)^c\big)\\ \notag
    &= P(r_mA_m+y_m; B_{r_m/2}(y_m+r_me_n)) - P(B_{r_m}(y_m); B_{r_m/2}(y_m+r_me_n))\\ \label{eq:MC 2}
    &=r_m^{n-1}\big[ P(A_m;B_{1/2}(e_n)) - P(B_1(0);B_{1/2}(e_n))\big] \leq \Lambda |\sigma_m| / r_m \overset{m\to \infty}{\to} 0 \,.
\end{align}
Let us now define the $(2,1)$-clusters $\X_m$ via 
\begin{align}\notag
 &\X_m(1) = \X(1)\,, \quad \X_m(2) = \big(\X(2) \cap B_R(0)\big) \cup \big((r_mA_m+y_m) \setminus B_R(0)\big)\,, \\ \notag
 &\X_m(3) = \big(\X(3) \cap B_R(0)\big) \cup \big(B_R(0)^c \setminus (r_mA_m+y_m) \big)\,.
\end{align}
Due to the definition of $\X_m$, \eqref{eq:vol change 2}, and \eqref{eq:vol change 0}, we may compute
\begin{align}\notag
    |\X_m(2) | &= |\X(2) \cap B_R(0)| + |(r_m A_m + y_m) \setminus B_R(0)| \\ \notag
    &=|\X(2) \cap B_R(0)|+ |\xbub(2)\setminus B_R(0)|+\sigma_m \\ \notag
    &=|\X(2) \cap B_R(0)|+ |\xbub(2)\setminus B_R(0)| + |\xbub(2) \cap B_R(0)| - |\X(2) \cap B_R(0)| \\ \notag
    &=|\xbub(2)| \,.  
\end{align}
Therefore, ${\bf m}(\X_m) = (1,m,\infty)$. By $\X(j) \Delta \xl(j) \cc B_R(0)$ for each $j$ and \eqref{eq:agrees with xbub}, 
\begin{align}\notag
    (\X(2)^\one \Delta (r_mA_m+y_m) ) \cap \pa B_R(0) = (\xl(2)^\one \Delta \xbub(2)^\one) \cap \pa B_R(0)\,,
\end{align}
so that by the local Hausdorff convergence of $\xbub(2)$ to $\xl(2)$ from Lemma \ref{lemma:convergence of weighted double bubbles},
\begin{align}\label{eq:bdry term dies}
    \hn \big( (\X(2)^\one \Delta (r_mA_m+y_m)) \cap \pa B_R(0)\big) \overset{m\to \infty}{\to} 0\,.
\end{align}
Now, by $\pc(\xbub ; B_R(0))\to \pc(\xl;B_R(0))$ from Lemma \ref{lemma:convergence of weighted double bubbles} and the minimality of $\xbub$ tested against $\X_m$ (which is permissible since ${\bf m}(\X_m) = (1,m,\infty)$), we may write
\begin{align}\notag
    \pc (\xl ; B_R(0)) &= \lim_{m\to \infty}\pc(\xbub ; B_R(0)) = \lim_{m\to \infty}\pc(\xbub) - \pc (\xbub ; \rn \setminus  B_R(0)) \\ \label{eq:first xl min comp}
    &\leq \liminf_{m\to \infty} \pc(\X_m ) -  \pc (\xbub ; \rn \setminus B_R(0))\,.
\end{align}
We compute $\pc(\X_m)$ using \eqref{eq:cut and paste} and estimate $ -  \pc (\xbub ; \rn \setminus \cl B_R(0))$ using \eqref{eq:MC 2}, yielding
\begin{align}\notag
    \pc(\X_m) - \pc (\xbub ; \rn \setminus B_R(0)) &\leq \pc(\X;B_R(0)) + c_{23}\hn \big( (\X(2)^\one \Delta (r_mA_m+y_m)) \cap \pa B_R(0)\big) \\ \notag
    &\qquad + c_{23}P(r_mA_m + y_m ; \rn \setminus \cl B_R(0)) \\ \notag
    &\qquad - c_{23}P(r_mA_m + y_m ; \rn \setminus \cl B_R(0)) + c_{23}\Lambda |\sigma_m|/r_m \\ \label{eq:second xl min comp}
    &= \pc (\X;B_R(0))  + \mathrm{o}(1)\,,
\end{align}
where in the last line we have used \eqref{eq:bdry term dies}. Combining \eqref{eq:first xl min comp}-\eqref{eq:second xl min comp} and sending $m\to \infty$ gives the desired minimality inequality \eqref{eq:desired minimality inequality}. 
\end{proof}

For the remainder of this section we focus on the case where $c_{12}=c_{13}$. It will be convenient to rewrite the energy. Specifically, for a $(1,2)$-cluster $\X$ and energy $\pc(\cdot;B)$ corresponding to weights $c_{12}=c_{13}$ and $c_{23}$, we have
\begin{align}\label{eq:pc rewrite 2}
    \pc(\X;B) = c_{13} P(\X(1);B) + c_{23} \hn(\X(2,3) \cap B)
    \quad \textup{for every Borel $B\subset \rn$}\,.
\end{align}
Before presenting the proof of Theorem \ref{thm:double bubble}, we give a lemma on the ``symmetrization" of a $(1,2)$-cluster $\X$ that will be used several times. For a $(1,2)$-cluster $\X$, we define $\X^\ste$ via
\begin{align}\notag
    \X^\ste(1) = \X(1)^\ste\,,\quad \X^\ste(2)= \{x\in \rn : x_n > 0\}\setminus \X^\ste(1)\,,\quad \X^S(3) = \rn \setminus (\X^\ste(1) \cup \X^\ste(2))\,;
\end{align}
recall that $\X(1)^S$ denotes the Steiner symmetrization of $\X(1)$ over the plane $\{x\in \mathbb{R}^n:x_n=0\}$.

\begin{lemma}[Steiner inequality for $(1,2)$-clusters]\label{lemma:proj and symm}
    If $\bf c$ is a family of positive weights with $c_{12}=c_{13}$, $\X$ is a $(1,2)$-cluster in $\rn$, and there exists $a>0$ such that 
\begin{align}\label{eq:x2x3 convention}
    \mathbb{R}^{n-1}\times (a,\infty) \subset \X(2)^\one\,,\quad \mathbb{R}^{n-1}\times (-\infty,-a) \subset \X(3)^\one\,,
\end{align}  
    then, setting $E(\X):=\{\overline{x}\in \mathbb{R}^{n-1}:\hone(\X(1)_\ovx)=0 \}$,
\begin{enumerate}[label=(\roman*)]
\item $\X^S(2,3)\ehn E(\X) \times \{0\}$, and for any Borel $B\subset \rn$
\begin{align}\label{eq:sum proj and symm}
\pc(\X^\ste;B)= c_{13} P(\X(1)^\ste;B) + c_{23}\hn( (E(\X)\times \{0\} )\cap B) \,,\qquad \mbox{and}
\end{align}
\item for any $r>0$, 
\begin{align}\label{eq:steiner x ineq}
    P(\X(1);C_r) &\geq P(\X^\ste(1);C_r)\qquad\mbox{and} \\ \label{eq:proj x ineq}
    \hn(\X(2,3); C_r) &\geq \hn(E(\X) \cap B_r^{n-1}(\overline{0}))\,,
\end{align}
with equality in \eqref{eq:steiner x ineq} only if for $\hn$-a.e.~ $x\in B_r^{n-1}(\overline{0})$, $\X(1)_\ovx$ is an interval, and equality in \eqref{eq:proj x ineq} only if there exists $E'_r(\X)\ehn E(\X) \cap B_r^{n-1}(\overline{0})$ and $t:E'_r(\X) \to [-a,a]$ such that 
\begin{align}\label{eq:its piecewise flat}
\X(2,3) \cap C_r\overset{\mathcal{H}^{n-1}}&{=} \{(\overline{y},t(\overline{y})):\overline{y}\in E'_r(\X)\}\quad\mbox{and} \\ \label{eq:normal conclusion}
\nu_{\X(2)}(x) &= -e_n \quad \mbox{for $\hn$-a.e.~ $x\in \X(2,3) \cap C_r$}\,. 
\end{align}
\end{enumerate}
\end{lemma}

\begin{proof}
The proof is divided into steps, with item $(i)$ in step one and item $(ii)$ further broken up into steps two through four.

\medskip

\noindent{\it Step one}: Here we prove $(i)$. We begin with \eqref{eq:pc rewrite 2} for $\X^\ste$, which reads
\begin{align}\notag
\pc(\X^\ste;B)= c_{13}P(\X^\ste(1);B) + c_{23}\hn(\pa^* \X^\ste(2) \cap \pa^* \X^\ste(3)\cap B)\,.
\end{align}
Therefore, $(i)$ would follow from the $\hn$-equivalence
\begin{align}\label{eq:hn equivalence}
    \pa^* \X^\ste(2) \cap \pa^* \X^\ste(3) \ehn E(\X) \times \{0\} \,.
\end{align}
Towards showing \eqref{eq:hn equivalence}, we first observe that by $\pa^* \X^S(j) \ehn \X^S(j)^{\scriptsize{(1/2)}}$ for $j=2,3$ (Federer's theorem) and the containments $\X^S(2)^\one\subset \{x_n>0\}$ and $\X^S(3)^\one \subset \{x_n<0\}$, 
\begin{align}\label{eq:first intermediate equiv}
    \pa^* \X^\ste(2) \cap \pa^* \X^\ste(3)  = \pa^* \X^\ste(2) \cap \pa^* \X^\ste(3) \cap \{x\in \rn: x_n = 0 \}\,.
\end{align}
By appealing to Lemma \ref{lemma:slice}, let $F\subset \mathbb{R}^{n-1}$ be a set such that $\hn(F^c)=0$ and if $\ovx \in F$, then $(i)$-$(iv)$ from that lemma hold for both $\X$ and $\X^S$. Applying in order $\hn (F^c)=0$, a rephrasing of the set intersection, and Lemma \ref{lemma:slice}.$(iii)$, we have
\begin{align}\notag
    \pa^* \X^\ste(2) \cap \pa^* \X^\ste(3) \cap \{x\in \rn: x_n = 0 \}\overset{\mathcal{H}^{n-1}}&{=} \pa^* \X^\ste(2) \cap \pa^* \X^\ste(3) \cap (F \times \{0\}) \\ \notag
    &= \{(\overline{x},0)\in   \pa^* \X^\ste(2) \cap \pa^* \X^\ste(3): \overline{x} \in F\} \\ \label{eq:second intermediate equiv}
    &= \{(\overline{x},0): 0 \in \pastr [\X^\ste(2)_\ovx]\cap \pastr [\X^\ste(3)_\ovx],\,\overline{x}\in F \}\,.
\end{align}
But by the definition of Steiner symmetrization,
\begin{align}\notag
    \pastr [\X^\ste(2)_\ovx] = \{\mathcal{H}^1(\X(1)_\ovx)/2 \}\quad \mbox{and}\quad   \pastr [\X^\ste(3)_\ovx] =\{-\mathcal{H}^1(\X(1)_\ovx)/2 \}\,,
\end{align}
so $0$ can only belong to both if $\mathcal{H}^1(\X(1)_\ovx)=0$. Thus the right hand side of \eqref{eq:second intermediate equiv} is 
\begin{align}\label{eq:third intermediate equiv}
    \{(\overline{x},0): \mathcal{H}^1(\X(1)_\ovx)=0 ,\,\overline{x}\in F \}\,.
\end{align}
Combining \eqref{eq:first intermediate equiv}-\eqref{eq:third intermediate equiv} and using $\hn(F^c)=0$ one last time, we arrive at \eqref{eq:hn equivalence} as desired.

\medskip

\noindent{\it Step two}: In steps two through four we prove $(ii)$. Since \eqref{eq:steiner x ineq} and the corresponding claim regarding equality cases is simply a restatement of the Steiner inequality \eqref{eq:steiner inequality} and the restriction that equality induces on the slices of $\X(1)$, we focus on the remainder of $(ii)$ involving $\X(2,3) \cap C_r$.

\medskip

As a preliminary, in this step we show that if $G\subset F$ is Borel measurable and $\hn(G)=0$, then 
\begin{align}\label{eq:g estimate}
    \hn\big(\X(2,3) \cap (G\times \mathbb{R}) \big)=0\,.
\end{align}
To prove \eqref{eq:g estimate}, it suffices to construct a sequence $\{G_j\}_j$ of subsets of $G$ with $\cup_j G_j=G$ and show that for any compact $K\cc \mathbb{R}$,
\begin{align}\label{eq:k estimate}
    \hn\big(\X(2,3) \cap (G_j \times K)\big)=0\qquad \forall j\,.
\end{align}
Let $G_j:=\{\overline{x}\in G: \mathcal{H}^0(\pastr [\X(2)_\ovx] \cap \pastr [\X(3)_\ovx] \cap K) \leq j\}$. By $G \subset F$ and Lemma \ref{lemma:slice}.$(i)$ (which guarantees the local finiteness of the perimeters of the slices), $\cup_j G_j = G$. To estimate the $\hn$-measure of $\X(2,3)\cap (G_j \times K)$, let $\p:\rn \to \mathbb{R}^{n-1}$ be the projection map $\p(x)=(x_1,\dots,x_{n-1})$. A direct computation shows that the tangential Jacobian of $\p$ with respect to the locally 
$\hn$-rectifiable set $\X(2,3)$ is given by
\begin{align}\notag
J^{\X(2,3)}\p(x) = \sqrt{\det \big(\nabla^{\X(2,3)} \p (x)^* \nabla^{\X(2,3)} \p(x)\big)} = |\nu_{\X(2)}(x)\cdot e_n|\quad\mbox{for $\hn$-a.e.~ $x\in \X(2,3)$}\,,
\end{align}
where $\nu_{\X(2)}(x)$ is the measure-theoretic outer unit normal to $\X(2)$. Using the fact that $\nu_{\X(2)}(x)\cdot e_n\neq 0$ on $G\times \mathbb{R}$ (from Lemma \ref{lemma:slice}.$(iv)$ and $G\subset F$), the area formula, and Lemma \ref{lemma:slice}.$(iii)$, we compute
\begin{align}\notag
  \int_{\X(2,3) \cap (G_j\times K)}1\,d\hn(x) &= \int_{\X(2,3)\cap (G_j\times K)} \frac{J^{\X(2,3)}\p(x)}{|\nu_{\X(2)}(x)\cdot e_n|}\,d\hn(x) \\ \notag
  &= \int_{G_j}\bigg(\int_{ \X(2,3)\cap (\{\overline{y}\} \times K)}\frac{1}{|\nu_{\X(2)}(x)\cdot e_n|} \,d\mathcal{H}^0(x)\bigg)\,d\mathcal{H}^{n-1}(\overline{y}) \\ \notag
  &\leq \int_{G_j} j\times \sup_{x\in (\{\overline{y}\}\times K) \cap \X(2,3)} \frac{1}{|\nu_{\X(2)}(x)\cdot e_n|}\,d\hn(\overline{y})=0\,,
\end{align}
where in the last equality we have used $\nu_{\X(2)}\cdot e_n \neq 0$ on $G \times \mathbb{R}$ and $\hn(G_j) \leq \hn(G)=0$ to conclude that the integral vanishes. This completes \eqref{eq:g estimate}.

\medskip

\noindent{\it Step three}: In the third step we prove the inequality \eqref{eq:proj x ineq}, with the equality case being analyzed in step four. Let $\X(2,3)^\parallel=\{x\in \X(2,3):\nu_{\X(2)}(x) \cdot e_n \neq 0\}$ and $\X(2,3)^\perp =\X(2,3) \setminus \X(2,3)^\parallel$. Then by the area formula and the fact that $\nu_{\X(2)}(x)\cdot e_n\neq 0$ on $F\times \mathbb{R}$,
\begin{align}\notag
  \int_{\X(2,3) \cap C_r} 1\,d\hn(x) &= \int_{\X(2,3)^\parallel\cap ([F\cap E(\X)]\times \mathbb{R}) \cap C_r} \frac{J^{\X(2,3)}\p(x)}{|\nu_{\X(2)}(x)\cdot e_n|}\,d\hn(x)\\ \notag
  &\qquad + \hn \big(\big[\X(2,3)^\perp \cup ((F^c\cup E(\X)^c)\times \mathbb{R})\big]\cap C_r\big) \\ \notag
  &= \int_{F \cap E(\X) \cap B_r^{n-1}(\overline{0})}\bigg(\int_{ \X(2,3)\cap(\{\overline{y}\}\times \mathbb{R})}\frac{1}{|\nu_{\X(2)}(x)\cdot e_n|}\,d\mathcal{H}^0(x)\bigg)\,d\hn(\overline{y}) \\ \label{eq:cons of area}
  &\qquad + \hn \big(\big[\X(2,3)^\perp \cup ((F^c\cup E(\X)^c)\times \mathbb{R})\big]\cap C_r\big)\,. 
\end{align}
Now if $\overline{y}\in F \cap E(\X) \cap B_r^{n-1}(\overline{0})$, then by the definitions of $F$ and $E(\X)$, $\X(2)_{\overline{y}}$ and $\X(3)_{\overline{y}}$ are sets of locally finite perimeter in $\mathbb{R}$ and
\begin{align}\label{eq:x1 not on slice}
\mathcal{H}^1(\X(1)_{\overline{y}})=0\,.
\end{align}
Furthermore, by \eqref{eq:x2x3 convention} and Lemma \ref{lemma:slice}.$(ii)$, 
\begin{align}\label{eq:ints in x2x3}
(a,\infty)\overset{\mathcal{H}^1}{\subset}\X(2)_{\overline{y}}^\oner\quad\mbox{and}\quad (-\infty,-a)\overset{\mathcal{H}^1}{\subset}\X(3)_{\overline{y}}^\oner\quad \forall \,\overline{y}\in F \cap E(\X) \cap B_r^{n-1}(\overline{0})\,.
\end{align}
The combination of \eqref{eq:x1 not on slice} and \eqref{eq:ints in x2x3} implies that $\pastr [\X(2)_{\overline{y}}] \cap \pastr [\X(3)_{\overline{y}}] \neq \varnothing$, which by Lemma \ref{lemma:slice}.$(iii)$ yields
\begin{align}\label{eq:reduced bdry projects}
    \pa^* \X(2) \cap \pa^* \X(3) \cap (\{y\}\times \mathbb{R})\neq \varnothing\quad \forall y\in F \cap E(\X) \cap B_r^{n-1}(\overline{0})\,.
\end{align}
Inserting this into \eqref{eq:cons of area}, we conclude that
\begin{align}\label{eq:cons of area 1a}
    \hn(\X(2,3) \cap C_r) &= \int_{F \cap E(\X) \cap B_r^{n-1}(\overline{0})}\bigg(\int_{ \X(2,3)\cap(\{\overline{y}\}\times \mathbb{R})}\frac{1}{|\nu_{\X(2)}(x)\cdot e_n|}\,d\mathcal{H}^0(x)\bigg)\,d\hn(\overline{y}) \\ \notag
  &\qquad + \hn \big(\big[\X(2,3)^\perp \cup ((F^c\cup E(\X)^c)\times \mathbb{R})\big]\cap C_r\big) \\ \label{eq:cons of area 2}
  &\geq \int_{F \cap E(\X) \cap B_r^{n-1}(\overline{0})} 1\,d\hn(\overline{y})=  \hn(F \cap E(\X) \cap B_r^{n-1}(\overline{0}))\,.
\end{align}
The inequality \eqref{eq:cons of area 2} is equivalent to \eqref{eq:proj x ineq} since $\hn(F^c)=0$. 

\medskip

\noindent{\it Step four}: It remains to demonstrate the two necessary conditions \eqref{eq:its piecewise flat}-\eqref{eq:normal conclusion} for equality in \eqref{eq:proj x ineq}. If equality holds in \eqref{eq:proj x ineq}, then it holds in \eqref{eq:cons of area 2} since $\hn(F^c)=0$. Inspecting \eqref{eq:cons of area 2}, we see that equality forces several conditions:
\begin{align}\notag
  &\mbox{$(a)$ for $\hn$-a.e.~ $\overline{y}\in F \cap E(\X) \cap B_r^{n-1}(\overline{0})$, $\X(2,3)\cap (\{\overline{y}\}\times \mathbb{R})$ is a singleton $\{(\overline{y},t(\overline{y}))\}$}\,, \\ \notag
  &\mbox{$(b)$ for these $\overline{y}\in F \cap B_r^{n-1}(\overline{0})$ such that $\mathcal{H}^0\big(\X(2,3)\cap (\{\overline{y}\}\times \mathbb{R})\big)=1$, $\nu_{\X(2)}((\overline{y},t(\overline{y})))=\pm e_n$, and } \\ \notag
  &\mbox{$(c)$ $\hn \big(\big[\X(2,3)^\perp \cup ((F^c\cup E(\X)^c)\times \mathbb{R})\big]\cap C_r\big)=0$}\,.
\end{align}
By $(c)$ and the fact that $\nu_{\X(2)}(x) \cdot e_n\neq 0$ on $F \times \mathbb{R}$,
\begin{align}\notag
    \hn(\X(2,3) \cap C_r) &= \hn \big(\X(2,3)^\parallel \cap \big([F\cap E(\X) ]\times \mathbb{R}\big)\cap C_r\big)\\ \label{eq:F cap E gobbles everything}
    &=\hn \big(\X(2,3) \cap \big([F\cap E(\X) ]\times \mathbb{R}\big)\cap C_r\big)\,.
\end{align}
Also, due to $(a)$ and Lemma \ref{lemma:slice}.$(iii)$, for $\hn$-a.e.~ $\overline{y}\in F \cap E(\X)\cap B_r^{n-1}(\overline{0})$, $\pastr [\X(2)_{\overline{y}}] \cap \pastr [\X(2)_{\overline{y}}]$ consists of the single point $t(\overline{y})$. In addition, by \eqref{eq:ints in x2x3}, $t(\overline{y}) \in [-a,a]$ and
\begin{align}\label{eq:t normal is down}
\nu_{\X(2)_{\overline{y}}}(t(\overline{y}))=-1 \quad \mbox{for $\hn$-a.e.~ $\overline{y}\in F \cap E(\X)\cap B_r^{n-1}(\overline{0})$}\,.    
\end{align}
So then using in order \eqref{eq:t normal is down}, Lemma \ref{lemma:slice}.$(iv)$, and $(b)$, we have
\begin{align}\notag
-1 &= \nu_{\X(2)_{\overline{y}}}(t(\overline{y})) = \frac{\nu_{\X(2)}(\overline{y},t(\overline{y}))\cdot e_n}{ |\nu_{\X(2)}(\overline{y},t(\overline{y}))\cdot e_n|} \\ \label{eq:normal points down} 
&= \nu_{\X(2)}(\overline{y},t(\overline{y}))\cdot e_n \quad \mbox{for $\hn$-a.e.~ $\overline{y}\in F \cap E(\X) \cap B_r^{n-1}(\overline{0})$}\,.
\end{align}
Therefore, by $(a)$ and \eqref{eq:normal points down}, we may choose Borel measurable $E_r'(\X)\subset F \cap E(\X) \cap B_r^{n-1}(\overline{0})$ with
\begin{align}\label{eq:eprime sees it all}
\hn(F \cap E(\X) \cap B_r^{n-1}(\overline{0}) \setminus E'_r(\X))=0
\end{align}
such that
\begin{align}\label{eq:its piecewise flat 3}
\X(2,3) \cap C_r \cap (E_r'(\X)\times \mathbb{R})&= \{(\overline{y},t(\overline{y})):\overline{y}\in E'_r(\X)\}\quad\mbox{and} \\ \label{eq:normal conclusion 3}
\nu_{\X(2)}(x) &= -e_n \quad \mbox{for $\hn$-a.e.~ $x\in \X(2,3) \cap C_r \cap (E_r'(\X)\times \mathbb{R})$}\,. 
\end{align}
Since $\hn(F \cap E(\X) \cap B_r^{n-1}(\overline{0}) \setminus E'_r(\X))=0$, by our preliminary estimate \eqref{eq:g estimate} and \eqref{eq:its piecewise flat 3} we find
\begin{align}\notag
   \hn \big(\X(2,3) \cap \big([F\cap E(\X) ]\times \mathbb{R}\big)\cap C_r\big) &= \hn \big(\X(2,3) \cap (E'_r(\X) \times \mathbb{R})\cap C_r\big) \\ \label{eq:f' sees all} 
   &= \hn\big(\X(2,3) \cap \{(\overline{y},t(\overline{y})):\overline{y}\in E'_r(\X)\}\big)\,.
\end{align}
Finally, \eqref{eq:f' sees all} and \eqref{eq:F cap E gobbles everything} guarantee that
\begin{align}\label{eq:e'r sees all}
    \hn( \X(2,3)\cap (E_r'(\X)^c \times \mathbb{R}) \cap C_r) = 0\,,
\end{align}
in which case \eqref{eq:its piecewise flat}-\eqref{eq:normal conclusion} follow from \eqref{eq:its piecewise flat 3}-\eqref{eq:normal conclusion 3}.
\end{proof}

\begin{remark}[Alternative proof of minimality]\label{remark:alternate proof of minimality}
    When $c_{12}=c_{13}$, the symmetrization inequalities from Lemma \ref{lemma:proj and symm} can be combined with the characterization of minimizing liquid drops in Theorem \ref{thm:liquid drops} to give an short proof of the $\bf c$-local minimality of $\xl$.
\end{remark}

\begin{proof}[Proof of Theorem \ref{thm:double bubble}]
The minimality of $\xl$ has already been proved in Theorem \ref{thm:local minimality of lens}. The uniqueness proof when $c_{12}=c_{13}$ is divided into steps, and the main outline is to combine the various types of rigidity contained in \eqref{eq:asymptotic exp 2} (the asymptotic expansion), Lemma \ref{lemma:proj and symm} (the ``flatness" in \eqref{eq:its piecewise flat}-\eqref{eq:normal conclusion}), Theorem \ref{thm:liquid drops} (characterization of minimizing liquid drops), and Theorem \ref{thm:symm} (rigidity in the Steiner inequality) to conclude that up to rotation and translation, $\X=\xl$. 

\medskip

\noindent{\it Step one}: Let $\X$ be a $\bf c$-locally minimizing $(1,2)$-cluster in $\rn$, with the added assumption that $\X$ has planar growth at infinity if $n\geq 8$. In this step we collect some basic properties of $\X$. 
    
\medskip
    
    First, by Lemma \ref{lemma:density estimates}, there exists $R>0$ such that $\X(1)\cc B_R(0)$. Second, by the boundedness of $\X(1)$ and the planar growth of $\X$ at infinity in $\rn$ for $n\geq 8$, we may apply Corollary \ref{corollary:asymptotic expansion} to find $R_0>R$, such that, up to a rotation and translation, $\X(2,3)\setminus C_{R_0}$ coincides with the graph of a solution $u:\mathbb{R}^{n-1}\setminus B_{R_0}^{n-1}(\overline{0})$ to the minimal surface equation satisfying the expansion at infinity
\begin{equation}\label{eq:asymptotic exp 2}
  u(\ovx) = \begin{cases}
  0 & \overline{x}\in \mathbb{R}^1,\,\, n=2\\ 
    0 + \displaystyle\frac{b}{|\overline{x}|^{n-3}}+ \displaystyle\frac{\overline{c}\cdot \overline{x}}{|\overline{x}|^{n-1}}+ \mathrm{O}(|\overline{x}|^{1-n})  & \overline{x}\in \R^{n-1},\,\,n\geq 3 \\ 
\end{cases}
\end{equation}
for some $b \in \mathbb{R}$, which we assume is $0$ when $n=3$, and $\overline{c} \in \mathbb{R}^{n-1}$. 

\medskip

Next, we claim that there exists $a>0$ such that,
\begin{align}\label{eq:containment in strip main proof}
    \cl \X(1) \cup \X(2,3) \cc \mathbb{R}^{n-1}\times [-a,a]\,.
\end{align}
To see this, note that for large $r>R_0$, the rectifiable varifold $V_r=\var (\X(2,3)\cap C_r, 1)$ is stationary in the open set $\rn \setminus \big[\cl \X(1) \cup \big(\pa B_r^{n-1}(\overline{0})\times [-1,1]\big)\big]$. Since $\cl \X(1) \cup (\pa B_r^{n-1}(\overline{0})\times [-1,1])$ is compact, we can apply the convex hull property \cite[Theorem 19.2]{Sim83} to deduce that
\begin{align}\label{eq:convex hull}
    \spt\, V_r = \cl (\X(2,3) \cap C_r) \subset \mathrm{Conv}\, \big( \cl \X(1) \cup \big(\pa B_r^{n-1}(\overline{0})\times [-1,1]\big)\big)\,,
\end{align}
where $\mathrm{Conv}$ denotes the convex hull. Let $a\geq 1$ be large enough so that $\cl \X(1) \cc \mathbb{R}^{n-1}\times [-a,a]$. Then sending $r\to \infty$ in \eqref{eq:convex hull} yields \eqref{eq:containment in strip main proof}. As a corollary, up to interchanging $\X(2)$ and $\X(3)$, 
\begin{align}\label{eq:x2x3 convention 2}
    \mathbb{R}^{n-1}\times (a,\infty) \subset \X(2)^\one\,,\quad \mathbb{R}^{n-1}\times (-\infty,-a) \subset \X(3)^\one\,.
\end{align} 
This will allow us to freely use Lemma \ref{lemma:proj and symm} on $\X$ and any compactly supported variation of $\X$, since \eqref{eq:x2x3 convention 2} is the only assumption on $\X$ in the statement of that lemma.

\medskip

\noindent{\it Step two}: In this step we prove via a comparison argument that $\X$ achieves equality in \eqref{eq:steiner x ineq} and \eqref{eq:proj x ineq} for all $r>R_0$. Let us begin by noting that by $\X(1)\cc B_R(0)$ and \eqref{eq:containment in strip main proof}, 
\begin{align}\label{eq:ex containment}
    E(\X):=\{\overline{x}\in \mathbb{R}^{n-1}:\hone(\X(1)_\ovx)=0 \} \supset \mathbb{R}^{n-1}\setminus B_{R}^{n-1}(\overline{0})\,.
\end{align}
Thus by \eqref{eq:x2x3 convention 2}-\eqref{eq:ex containment} and the fact that $\hn$-measure decreases upon projection onto $\{x_n=0\}$,
\begin{align}\notag
  \hn (\X(2,3) \cap  C_r \setminus C_s) &\geq\omega_{n-1}(r^{n-1}-s^{n-1})  \\ \label{eq:increasing 1}
  &= \hn(E(\X) \cap B_r^{n-1}(\overline{0}) \setminus B_s^{n-1}(\overline{0}) ) \qquad \forall R<R_0<s<r\,.
\end{align}
Therefore, the function $\ee:(R_0,\infty)\to \mathbb{R}$ defined by
\begin{equation}\notag
   \ee(r)= c_{13}\big[P(\X(1))-P(\X^S(1)) \big]+ c_{23}\big[\hn(\X(2,3) \cap C_r)- \hn(E(\X) \cap B_r^{n-1}(\overline{0}))\big] 
\end{equation}
satisfies
\begin{align}\label{eq:excess increases}
    \mbox{$\ee(r)$ is increasing for $r>R_0$}\,,
\end{align}
by \eqref{eq:increasing 1}, and, by \eqref{eq:steiner x ineq}-\eqref{eq:proj x ineq},
\begin{align}\label{eq:excess is positive}
    0 \leq \ee(r) \quad \forall r>R_0\,.
\end{align}

\medskip

We show now that $\ee(r)=0$ for all $r>R_0$. Let us first consider the case that $n\geq 3$.  For each $r>R_0$, let $\Pi_r$ be the plane
$$
\Pi_r = \{(\ovx, \overline{c} \cdot \ovx / r^{n-1}): \ovx \in \mathbb{R}^{n-1}\}\,,
$$
and let $H_r^+$ and $H_r^-$ be the halfspaces with common boundary $\Pi_r$ such that $\nu_{H_r^+}\cdot e_n<0$. $\Pi_r$ is chosen in this fashion because on $\partial C_r$, it is very close to $\X(2,3)$. More precisely, by the definition of $\Pi_r$ and since $\X(2,3)\setminus C_R$ is the graph of $u$, for $r>R_0$,
\begin{align}\notag
  \pa C_r \cap ((H_r^+)^\one \Delta \X(2)^\one) \subset \big\{(\ovx,t):\ovx\in \pa B_r^{n-1}(\overline{0}),\,\, |t - u(\ovx)| \leq |u(\ovx) - \overline{c}\cdot \overline{x}/r^{n-1}| \big\}\,.
\end{align}
By the asymptotic expansion \eqref{eq:asymptotic exp 2}, we may thus estimate
\begin{align}\label{eq:pa cr estimate}
    \hn \big( \pa C_r \cap ((H_r^+)^\one \Delta \X(2)^\one) \big) \leq \omega_{n-2}r^{n-2}\mathrm{O}(r^{1-n}) = \mathrm{O}(r^{-1})\,.
\end{align}
Next, let $T_r$ be a rotation matrix such that $T_r(\{x_n=0\})=\Pi_r$, and define a new $(1,2)$-cluster $\X_r$ by replacing $\X(j) \cap C_r$ with $T_r(\X^S(j))\cap C_r$ for each $j$, so
\begin{align}\notag
    \X_r(1) &= T_r(\X^S(1))\,,\\   \notag
     \X_r(2)&= \big( T_r(\X^S(2)) \cap C_r\big) \cup (\X(2) \setminus C_r)= \big(H_r^+ \setminus T_r(\X^S(1)) \cap C_r\big) \cup (\X(2) \setminus C_r)\,, \\ \notag
      \X_r(3)&= \big( T_r(\X^S(3)) \cap C_r\big) \cup (\X(3) \setminus C_r)= \big(H_r^- \setminus T_r(\X^S(1)) \cap C_r\big) \cup (\X(3) \setminus C_r)\,.
\end{align}
Then for large $r$, $\X(j) \Delta \X_r(j) \cc B_{r+1}^{n-1}(\overline{0})\times (-a-1,a+1)$. Also, recalling the characterization of $\X^S(2,3)$ from Lemma \ref{lemma:proj and symm}.$(i)$, we have
\begin{align}\label{eq:rotated char}
    \pa^* \X_r(2) \cap \pa^* \X_r(3) \cap C_r = T_r( \X^S(2,3)) \cap C_r =  T_r\big(E(\X) \times \{0\} \big) \cap C_r\,.
\end{align}
By the $\bf c$-local minimality of $\X$, \eqref{eq:rotated char}, and the fact that $\X$ and $\X_r$ agree outside $C_r$, we estimate
\begin{align}\notag
   0&\leq \pc(\X_r;B_{r+1}^{n-1}(\overline{0})\times (-a-1,a+1)) - \pc(\X;B_{r+1}^{n-1}(\overline{0})\times (-a-1,a+1)  \\ \notag
   &= c_{13}P(T_r(\X^S(1))) + c_{23}\hn\big(T_r\big(E(\X) \times \{0\}\big) \cap C_r \big) + c_{23}P(\X_r(2);\pa C_r) \\ \notag
   &\qquad - c_{13}P(\X(1)) - c_{23}\hn (\X(2,3) \cap C_r) \\ \notag
   &= c_{13}\big[P(\X^S(1)) - P(\X(1)) \big] +  c_{23}\hn\big(T_r\big(E(\X) \times \{0\}\big) \cap C_r \big) + c_{23}P(\X_r(2);\pa C_r) \\ \label{eq:big cr calc} 
   &\qquad - c_{23}\hn (\X(2,3) \cap C_r)\,.
\end{align}
By using in order $T_r(\{x_n=0\}) = \Pi_r$ and the definition of $E(\X)$, the definition of $\Pi_r$ and the ``cut-and-paste" formula \eqref{eq:cut and paste} along $\pa C_r$, $\sqrt{1+t^2} \leq 1+t^2$ and \eqref{eq:pa cr estimate}, and $n\geq 3$, we bound the second, third, and fourth terms from above:
\begin{align}\notag
  &\hspace{-1cm}c_{23}\hn\big( T_r\big(E(\X) \times \{0\}\big) \cap C_r \big) + c_{23}P(\X_r(2);\pa C_r)- c_{23}\hn (\X(2,3) \cap C_r) \\ \notag
  &=c_{23}\hn( \Pi_r \cap C_r)  - c_{23}\hn\big(T_r(\{\overline{x}:\mathcal{H}^1(\X(1)_\ovx)>0\})\big) + c_{23}P(\X_r(2);\pa C_r) \\ \notag
  &\qquad - c_{23}\hn (\X(2,3) \cap C_r) \\ \notag
  &= c_{23}\int_{B_r^{n-1}(\overline{0})}\sqrt{1+|\overline{c}|^2/r^{2n-2} }\,d\overline{x} - c_{23}\hn\big(\{\overline{x}:\mathcal{H}^1(\X(1)_\ovx)>0\}\big)\\ \notag
  &\qquad + c_{23}\hn\big((\X(2)^\one \Delta (H_r^+)^\one )\cap \pa C_r\big) - c_{23}\hn (\X(2,3) \cap C_r)  \\ \notag
    &\leq c_{23} \omega_{n-1}r^{n-1} + c_{23}|\overline{c}|^2\omega_{n-1}r^{n-1+2-2n} - c_{23}\hn\big(\{\overline{x}:\mathcal{H}^1(\X(1)_\ovx)>0\}\big) \\ \notag
    &\qquad + \mathrm{O}(r^{-1})  - c_{23}\hn (\X(2,3) \cap C_r)\\ \label{eq:Cr comp}
    &= c_{23}\hn(E(\X) \cap B_r^{n-1}(\overline{0}) ) - c_{23}\hn (\X(2,3) \cap C_r) +  \mathrm{O}(r^{-1})\,.
\end{align}
Plugging the bound \eqref{eq:Cr comp} into \eqref{eq:big cr calc} for $r>R_0$, we arrive at
\begin{align}\notag
   0\leq  c_{13}\big[P(\X^S(1))- P(\X(1)) \big]+ c_{23}\big[ \hn(E(\X) \cap B_r^{n-1}(\overline{0}))-\hn(\X(2,3) \cap C_r)\big]+ \mathrm{O}(r^{-1})\,.
\end{align}
After moving the first two terms to the other side, this inequality says in terms of $\ee(r)$ that
\begin{align}\label{eq:excess vanishes}
   \ee(r) \leq  \mathrm{O}(r^{-1})\,.
\end{align}
In summary, by combining \eqref{eq:excess increases}, \eqref{eq:excess is positive}, and \eqref{eq:excess vanishes}, we have shown that for $n\geq 3$, $\ee(r)$ is increasing, non-negative, and bounded from above by $\mathrm{O}(r^{-1})$ as $r\to \infty$ --  so $\ee(r)$ must be $0$ for all $r>R_0$ if $n\geq 3$. The same conclusion when $n=2$ follows by testing the minimality of $\X$ against $\X^S$, which is acceptable by \eqref{eq:asymptotic exp 2}, to find $\mathbf{e}(r) \leq 0$, and recalling that $\mathbf{e}(r) \geq 0$ from \eqref{eq:excess is positive}. In fact, since \eqref{eq:steiner x ineq}-\eqref{eq:proj x ineq} guarantee that both the terms in $\ee(r)$ are each non-negative for $r>R_0$, $\ee(r) \equiv 0$ on $(R_0,\infty)$ actually gives the finer information 
\begin{align}\label{eq:finer info 1}
   &P(\X(1))= P(\X^S(1))\quad\mbox{and}  \\ \label{eq:finer info 2}
   &\hn(\X(2,3) \cap C_r) = \hn(E(\X) \cap B_r^{n-1}(\overline{0}))\qquad \forall r>R_0\,. 
\end{align}

\medskip

\noindent{\it Step three}: Let $\rl$ be the radius of the disk $\xl(1)^\one \cap \{x_n=0\}$. Here we show that \eqref{eq:finer info 1}-\eqref{eq:finer info 2} force
\begin{align}\label{eq:its a big sheet with a hole}
    \X(2,3) \overset{\hn}&{=} \{x_n = 0 \} \setminus \big(\cl B_{\rl}^{n-1}(0) \times \{0\}\big)\quad \mbox{and} \\ \label{eq:xs is the lens}
    \X^S &= \xl\quad \mbox{up to translations of $\X^S(1)$ along $\{x_n=0\}$ and $\mathcal{L}^n$-null sets}\,.
\end{align}
First, according to \eqref{eq:its piecewise flat}-\eqref{eq:normal conclusion} from Lemma \ref{lemma:proj and symm}, the equality \eqref{eq:finer info 2} for every $r>R_0$ entails the existence of $E'(\X)\ehn E(\X)$ and $t:E'(\X) \to [-a,a]$ such that 
\begin{align}\label{eq:its piecewise flat 2}
\X(2,3) \overset{\hn}&{=} \{(\overline{y},t(\overline{y})):\overline{y}\in E'(\X)\}\quad\mbox{and} \\ \label{eq:normal conclusion 2}
\nu_{\X(2)}(x) &= -e_n \quad \mbox{for $\hn$-a.e.~ $x\in \X(2,3)$}\,. 
\end{align}
By a simple argument using the divergence theorem \cite[Lemma 2.2]{SteZie94}, the fact that $\nu_{\X(2)}=-e_n$ $\hn$-a.e.~ on $\X(2,3)$ implies that for every $x\in \X(2,3)$ and $B_r(x)$ with $|B_r(x) \setminus (\X(2) \cup \X(3))|=0$, which ensures that $\pa^* \X(2) \cap B_r(x) = \X(2,3)$, we have
\begin{align}\notag
\pa^* \X(2)\cap B_r(x)=P_x \cap B_r(x)\quad \mbox{for the plane $P_x\ni x$ with normal $-e_n$}\,.
\end{align}
Thus for any compact $K\cc \rn$ such that $\cl \X(1) \subset K$ and $\rn\setminus K$ is connected, 
\begin{align}\label{eq:outside is a plane}
\X(2,3) \setminus K = \{ x_n=0\} \setminus K\,,
\end{align}
where we have used \eqref{eq:asymptotic exp 2} (which says that asymptotically $\X(2,3)$ is $\{x_n=0\}$) to eliminate the possiblity that $\X(2,3)$ is a vertical translation of $\{x_n=0\}$. In particular,
\begin{align}\notag
    \X(j) \Delta \xl(j) \subset B_{R_0}(0) \quad \forall 1\leq j \leq 3\,.
\end{align}
Now we can test the minimality of $\X$ against $\xl$ on $C_r$ for any $r>R_0$ to obtain
\begin{align}\label{eq:we tested it}
    \pc (\X ; C_r) \leq \pc (\xl ; C_r)\,.
\end{align}
By \eqref{eq:energy of symmetric cluster} applied to $\X^S$, \eqref{eq:finer info 1}-\eqref{eq:finer info 2}, \eqref{eq:we tested it}, and then \eqref{eq:energy of symmetric cluster} applied to $\xl$, for any $r>R_0$
\begin{align}\notag
2c_{13}\mathcal{F}_{c_{23}/(2c_{13})}(\X^S(1) \cap H;H) +  c_{23}\omega_{n-1}r^{n-1} &=\pc (\X^S;C_r) = \pc(\X; C_r)\\ \notag
&\leq \pc(\xl;C_r) \\ \notag
&= 2c_{13}\mathcal{F}_{c_{23}/(2c_{13})}(\xl(1) \cap H;H) +  c_{23}\omega_{n-1}r^{n-1} \,.
\end{align}
But by Theorem \ref{thm:liquid drops}, $\xl \cap \{x_n>0 \}$ is the unique, volume-constrained minimizer for $\mathcal{F}_{c_{23}/(2c_{13})}$ up to horizontal translations and $\mathcal{L}^n$-null sets. This immediately yields \eqref{eq:xs is the lens}.

\medskip

Turning now towards \eqref{eq:its a big sheet with a hole}, let us ignore the horizontal translation and null sets, so that $\xl=\X^S$. We define the function $L_{\X(1)}:\mathbb{R}^{n-1} \to [0,\infty)$ by
\begin{align}\notag
    L_{\X(1)}(\ovx) = \mathcal{H}^1(\X(1)_\ovx) = \mathcal{H}^1(\X^S(1)_\ovx)\,.
\end{align}
Since $\X^S = \xl$, the precise representative of $L_{\X(1)}$ (see Subsection \ref{subsec:sfop prelim}) is given by
\begin{align}\label{eq:prec rep of L}
    L_{\X(1)}^*(\ovx) = \mathcal{H}^1( \xl(1)_\ovx) =: L_{\rm lens}(\ovx)\,.
\end{align}
Therefore
\begin{align}\label{eq:whats ex}
E'(\X) \ehn E(\X) \ehn \mathbb{R}^{n-1} \setminus B_{\rl}^{n-1}(\overline{0})\quad \mbox{and thus }\quad |\X(1) \setminus C_{\rl}|=0 \,. 
\end{align}
By the volume density estimate \eqref{eq:lower volume bound}, this yields $\cl \X(1) \setminus \cl C_{\rl} = \varnothing$. We can therefore choose $r'>0$ such that $\cl \X(1) \subset \cl B_{\rl}^{n-1}(\overline{0}) \times [-r',r']$. Since this set is compact and connected, \eqref{eq:outside is a plane} says that
\begin{align}\label{eq:whats x23}
  \X(2,3) \setminus \big( \cl B_{\rl}^{n-1}(\overline{0}) \times [-r',r']\big) = \{x_n=0\} \setminus \big(\cl B_{\rl}^{n-1}(0) \times \{0\}\big)\,.
\end{align}
Combined with \eqref{eq:whats ex} and \eqref{eq:its piecewise flat 2}, \eqref{eq:whats x23} implies \eqref{eq:its a big sheet with a hole}.

\medskip

\noindent{\it Step four}: Finally we conclude the proof of uniqueness by showing that $\xl = \X$ up to the rotations, translations, and null sets we have accrued thus far. By the fact that $P(\X^S(1); \partial C_r)=P(\xl(1);\pa C_r)=0$, \eqref{eq:steiner inequality},  and \eqref{eq:finer info 1}, 
\begin{align}\notag
P(\X^S(1))&=P(\X^S(1);B_{\rl}^{n-1}(\overline{0})\times \mathbb{R})\leq P(\X(1) ; B_{\rl}^{n-1}(\overline{0})\times \mathbb{R}) \\ \notag
&\leq P(\X(1);B_{\rl}^{n-1}(\overline{0})\times \mathbb{R}) + P(\X(1);\partial B_{\rl}^{n-1}(\overline{0})\times \mathbb{R})= P(\X(1))= P(\X^S(1))\,.
\end{align}
Thus these are all equalities, and so 
\begin{align}\label{eq:equality over lens}
    P(\X^S(1);B_{\rl}^{n-1}(\overline{0})\times \mathbb{R}) = P(\X(1) ; B_{\rl}^{n-1}(\overline{0})\times \mathbb{R})\,.
\end{align}
Furthermore, by \eqref{eq:prec rep of L}
\begin{align}\label{eq:prec rep pos on B}
L^*_{\X(1)}(\ovx)=L_{\rm lens}(\ovx)>0\quad \forall \ovx \in  B_{\rl}^{n-1}(\overline{0})\,.
\end{align}
Lastly, by $\X^S(1) = \xl(1)$ and the explicit description of $\pa \xl(1)$ as two spherical caps,
\begin{align}\label{eq:no vertical parts 2}
   \mathcal{H}^{n-2}\big(\{x\in \pa^* (\X(1)^S) : \nu_{E^S}(x) \cdot e_n = 0 \}\cap (B_{\rl}^{n-1}(\overline{0})\times \mathbb{R})\big)=0
\end{align}
(in fact it is empty). We therefore have a connected $(n-1)$-dimensional set $\Omega = B_{\rl}^{n-1}(\overline{0})$, a set of finite perimeter $\X(1)$ achieving equality in the Steiner inequality \eqref{eq:steiner inequality} over $\Omega$ (by \eqref{eq:equality over lens}), and a distribution function $L^*_{\X(1)}$ that is strictly positive $\mathcal{H}^{n-2}$-a.e.~ on $\Omega$ (by \eqref{eq:prec rep pos on B}) and satisfies \eqref{eq:no vertical parts} (by \eqref{eq:no vertical parts 2}). These are exactly the assumptions of Theorem \ref{thm:symm}, which we apply then to conclude that $\X(1) \overset{\mathcal{L}^n}{=} \X^S(1)$ up to a vertical translation. However, we know that $\X(2,3)\ehn \{x_n = 0 \} \setminus \big(\cl B_{\rl}^{n-1}(0) \times \{0\}\big)$ by \eqref{eq:its a big sheet with a hole}, and it is impossible that the $(1,2)$-cluster $\X$ could have interfaces given by
\begin{align}\notag
  \X(1,2) \cup \X(1,3) \ehn  \pa^* \xl(1) + te_n\mbox{ for $t\neq 0$}\,, \quad \X(2,3) \ehn \{x_n = 0 \} \setminus \big(\cl B_{\rl}^{n-1}(0) \times \{0\}\big)\,.
\end{align}
So $\X(1) \overset{\mathcal{L}^n}{=} \X^S(1)$, which together with \eqref{eq:its a big sheet with a hole}-\eqref{eq:xs is the lens} gives 
\begin{align}\notag
 \X(j)\overset{\mathcal{L}^n}{=}\xl(j) \quad \forall \, 1\leq j \leq 3\,.   
\end{align}
The proof of Theorem \ref{thm:double bubble} is complete.
\end{proof}

\begin{remark}
    As mentioned in Remark \ref{remark:residue connection}, the exterior minimal surface $\X(2,3)$ is similar to the exterior minimal surfaces arising from large volume isoperimetry in \cite{MagNov22}. The usage of the convex hull property in step one, the gluing onto tilted planes in step two, and the monotonicity of the ``cylindrical energy gap" have natural analogues in \cite[Steps 2-3, proof of Theorem 1.6]{MagNov22} and \cite[Step 2, proof of Theorem 1.1]{MagNov22}.
\end{remark}

\bigskip

\noindent{\bf Acknowledgments.} LB acknowledges support from an NSERC (Canada) Discovery Grant.

\bibliographystyle{abbrv}
\bibliography{references}
\end{document}